\documentclass[oneside,12pt]{amsart}
\usepackage{amssymb, mathrsfs, amscd}
\usepackage[all]{xy}

\usepackage{graphics}
\usepackage{graphicx}

\newtheorem*{thma}{Theorem A}
\newtheorem*{thmb}{Theorem B}
\newtheorem*{thmc}{Theorem C}
\newtheorem*{ques}{Question}

\newtheorem*{conv}{Convention}
\newtheorem{theorem}{Theorem}[section]

\newtheorem{lemma}[theorem]{Lemma}

\newtheorem{proposition}[theorem]{Proposition}
\newtheorem{corollary}[theorem]{Corollary}

\newtheorem{remark}[theorem]{Remark}

\newtheorem{example}[theorem]{Example}

\newtheorem{notation}[theorem]{Notation}

\newcommand{\mF}{\mathcal {F}}
\newcommand{\Z}{\mathbb Z}
\newcommand{\F}{\mathbb F}
\newcommand{\G}{\Gamma}
\newcommand{\orb}{\mathcal{O}_{\mF}G}
\newcommand{\orbmod}{\mbox{Mod-}\mathcal{O}_{\mF}G}

\newcommand{\nathom}{\mathrm{Hom}_{\mF}}

\title[Geometric dimension of groups]{Geometric dimension of groups for the family of virtually cyclic subgroups}

\author{Dieter Degrijse}
\address{Department of Mathematics, KU Leuven, Kortrijk, Belgium}%
\email{Dieter.Degrijse@kuleuven-kulak.be}%

\author{Nansen Petrosyan}
\address{Department of Mathematics, KU Leuven, Kortrijk, Belgium}%
\email{Nansen.Petrosyan@kuleuven-kulak.be}%
\thanks{Both authors were supported by the Research Fund KU Leuven.}
\thanks{The second author was also supported by the FWO-Flanders Research Fellowship.}

\begin{document}
\maketitle
\begin{abstract} By studying commensurators of virtually cyclic groups, we prove that every elementary amenable group of finite Hirsch length $h$ and cardinality $\aleph_n$ admits a finite dimensional classifying space with virtually cyclic stabilizers of dimension $n+h+2$. We also provide a criterion for groups that fit into an extension with torsion-free quotient to admit a finite dimensional classifying space with virtually cyclic stabilizers. Finally, we exhibit examples of integral linear groups of type $F$  whose  geometric dimension for the family of virtually cyclic subgroups is finite but arbitrarily larger than the geometric dimension for proper actions. This answers a question posed by W.~L\"{u}ck. 
\end{abstract}
\tableofcontents
\section{Introduction}
A classifying space of a discrete group $G$ for a family of subgroups $\mathcal{F}$ is a terminal object in the homotopy category of $G$-CW complexes with stabilizers in $\mathcal{F}$ (see \cite{tom}). Such a space is also called a model for $E_{\mathcal{F}}G$.  A model for $E_{\mathcal{F}}G$  always exists for any given discrete group $G$ and a family of subgroups $\mathcal{F}$ but it need not be of finite type or finite dimensional.  
When $\mathcal{F}$ is the family of virtually cyclic subgroups, $E_{\mathcal{F}}G$ is denoted  by $\underline{\underline{E}}G$. Questions concerning finiteness properties of $\underline{\underline{E}}G$ have been especially motivated by the Farrell-Jones Isomorphism Conjecture which describes the structure of the algebraic  K- and L-theory of the group ring of $G$  by a certain equivariant homology theory of $\underline{\underline{E}}G$  (e.g.~see \cite{BartelsLuckReich},\cite{DL} and \cite{FJ}).

Finite dimensional models for $\underline{\underline{E}}G$ have been constructed for several interesting classes of groups, for example,  word-hyperbolic groups (Juan-Pineda, Leary, \cite{LearyPineda}), relatively hyperbolic groups (Lafont, Ortiz, \cite{LafontOrtiz}), virtually polycyclic groups (L\"{u}ck,Weiermann, \cite{LuckWeiermann}) and $\mathrm{CAT}(0)$-groups (Farley, \cite{Farley}, L\"{u}ck, \cite{Luck3}).

In \cite{DP}, we have shown that every elementary amenable group $G$ of finite Hirsch length and cardinality $\aleph_n$ admits a finite dimensional model for $\underline{\underline{E}}G$.
However, an explicit upper bound for the minimal dimension of a model for $\underline{\underline{E}}G$ was not given. In Section 4, we prove the following.
\begin{thma} \label{th: intro main theorem 1} Let $G$ be an elementary amenable group of finite
Hirsch length $h$ and cardinality $\aleph_n$. Then there exists a model for $\underline{\underline{E}}G$ of dimension $n+h+2$.
\end{thma}
Let us point that if an elementary amenable group $G$ has infinite Hirsch length then it cannot admit a finite dimensional model for $\underline{\underline{E}}G$. In Section 5, we  give examples of groups that show that the  bounds obtained in the theorem are attained for each positive value of the Hirsch length $h$ and any cardinality $\aleph_n$. Theorem A thus answers  problems 46.1 and 46.2  stated by Lafont in \cite{Lafont} for the class of elementary amenable groups of cardinality $\aleph_n$.

The smallest possible dimension of a model for $\underline{\underline{E}}G$ is  called the geometric dimension of $G$ for the family of virtually cyclic subgroups, denoted by $\underline{\underline{\mathrm{gd}}}(G)$.  In Section 6, we address the following basic question.
\begin{ques} If $\Gamma$ fits into a short exact sequence of groups $N \rightarrow \Gamma \to Q$ such that $\underline{\underline{\mathrm{gd}}}(N)$ and  $\underline{\underline{\mathrm{gd}}}(Q)$ are both finite, is $\underline{\underline{\mathrm{gd}}}(\Gamma)$ also finite?
\end{ques}

The answer to this question and the analogous  question for the family of finite subgroups  is unknown in general (see \cite{misl}, \cite[5.5]{Luck2}). We give a positive answer provided that the group $Q$ is torsion-free and the number $\delta(N)$ defined below is finite.

For a given  group $N$ of cardinality $\aleph_n$, we let $\mathcal{E}_N$ be the collection of all groups that can be formed as  an extension of a finitely generated subgroup of $N$ by a finite cyclic group.  We then  define \[\delta(N)= \sup\{\underline{\mathrm{gd}}(G)\ |  G \in \mathcal{E}_N\}+n,\]
where $\underline{\mathrm{gd}}(G)$ is the geometric dimension of $G$ for proper actions.
\begin{thmb} Let $N$ be a group of cardinality $\aleph_n$ and let
$N \rightarrow \Gamma \to Q$ be  a short exact sequence of groups
such that $Q$ is torsion-free. Then
 \[  \underline{\underline{\mathrm{gd}}}(\Gamma) \leq \underline{\underline{\mathrm{gd}}}(Q)+ \max\{\underline{\underline{\mathrm{gd}}}(N), \delta(N)\}+2. \]
\end{thmb} 
There are several types of groups $N$ for which  $\delta(N)$ is bounded by $\underline{\mathrm{gd}}(N)+1$ and it is a general fact that $\underline{\mathrm{gd}}(N)$ is at most as large as $\underline{\underline{\mathrm{gd}}}(N)+1$.  For example, if $N$ is a group with $\underline{\mathrm{gd}}(N)=1$, such as a virtually free group, then $\delta(N)\leq 1$ (see \cite[1.1-2]{Dunwoody}).  Also, if  $N$ is a finite extension of an orientable surface group, then $\delta(N)=\underline{\mathrm{gd}}(N)=2$ (see \cite[4.4(2)]{mislin}). Lastly, if $N$ is a countable elementary amenable group of finite Hirsch length $h(N)$, then $\delta(N)\leq h(N)+1$ and  $h(N)\leq \underline{\mathrm{gd}}(N)$ (see \cite{FloresNuc}). 

There are also groups $N$ that have $\delta(N)$  greater than $\underline{\mathrm{gd}}(N)+1$. In fact, in Example \ref{ex: counterex}, we show that there are groups for which $\delta(N)$ is finite but arbitrarily larger than $\underline{\mathrm{gd}}(N)$. A  general method of constructing  such groups is given in Theorem C. Let us note that we are not aware of an example of a group $N$ for which $\underline{\mathrm{gd}}(N)$ is finite but  $\delta(N)$ is infinite. In Remark \ref{remk: Question of Lafont}, we point out how an example of a group with this property may be used to construct a $2$-generated group satisfying $\underline{\mathrm{gd}}(\Gamma)<\infty$ and $\underline{\underline{\mathrm{gd}}}(\Gamma)=\infty$, thus answering a question of Lafont in \cite[46.4]{Lafont}. 

Recall that if $X$ is a property of a group, then a group is said to be of {type $VX$} if it contains a finite index subgroup of type $X$. Given a family of subgroups $\mathcal{F}$ of a group $G$, one can consider the {Bredon cohomological dimension} $\mathrm{cd}_{\mathcal{F}}(G)$ of the group $G$ for the family $\mathcal{F}$ (see Section 2), which satisfies \[\mathrm{cd}_{\mathcal{F}}(G) \leq  \mathrm{gd}_{\mathcal{F}}(G) \leq \max\{3, \mathrm{cd}_{\mathcal{F}}(G) \}.\] 
When $\mathcal{F}$ is the family of finite or virtually cyclic subgroups of the group $G$, then $\mathrm{cd}_{\mathcal{F}}G$ is denoted by $\underline{\mathrm{cd}}G$ or $\underline{\underline{\mathrm{cd}}}G$, respectively.
\begin{thmc}\label{th: lueck's question}Let $\{L_i \;|\; i=1, \dots ,s\}$ be a collection of finite acyclic flag complexes. Let $\{p_i \;|\; i=1, \dots ,s\}$ be a collection of primes and  $q$ be a prime that is not in this set. For each $i$, suppose the complex $L_i$ is of dimension $n_i\geq 1$ and the cyclic group $C_{p_i}$ acts admissibly and simplicially on $L_i$ such that $$\mathrm{H}_{n_i-1}(L_i^{C_{p_i}}, \F_q)\ne 0. $$
For each $i$, denote by $H_{i}$ the Bestvina-Brady group associated to $L_i$. Let $H_i\rtimes C_{p_i}$ be the semi-direct product formed by the automorphism action of $C_{p_i}$ on $H_i$  that is induced by the action of $C_{p_i}$ on $L_i$.
\begin{enumerate}
\item[(a)]  Let  $\Gamma_s=\prod_{i=1}^s H_{L_i}\rtimes C_{p_i}$. Then $\Gamma_s$ is an integral linear group of type $VFP$ such that $$\mathrm{vhd}(\Gamma_s)=\mathrm{vcd}(\Gamma_s)=\sum_{i=1}^s n_i \;\;\;\; \mbox{and} \;\;\;\; \underline{\mathrm{hd}}(\Gamma_s)=\underline{\mathrm{cd}}(\Gamma_s)=s+\sum_{i=1}^s n_i.$$
\item[(b)]  Suppose further that $s\geq 2$ and $p_i\ne p_j$ for all $1\leq i< j\leq s$. Let $\Lambda_s=(\prod_{i=1}^s H_{L_i})\rtimes \Z$  be the pull-back of $\Gamma_s$ under projection of the infinite cyclic group $\Z$ onto $\prod_{i=1}^s  C_{p_i}$.  Then $\Lambda_s$ is an integral linear group of type $FP$ such that $$\mathrm{hd}(\Lambda_s)=\mathrm{cd}(\Lambda_s)=1+\sum_{i=1}^s n_i \;\;\;\; \mbox{and} \;\;\;\; s+\sum_{i=1}^s n_i \leq \underline{\underline{\mathrm{cd}}}(\Lambda_s) < \infty.$$
\end{enumerate}
Moreover, if $L_i$ is contractible for each $i=1, \ldots ,s$, then $\Gamma_s$ is of type $VF$ and $\Lambda_s$ is of type $F$. 
\end{thmc}
A question posed by   L\"{u}ck in \cite{Luck3} (see also \cite{LuckWeiermann}) asks whether or not the inequality $\underline{\underline{\mathrm{gd}}}(\Gamma) \leq \underline{\mathrm{gd}}(\Gamma)+1$ holds for every group $\Gamma$. In Corollary \ref{cor: cor main theorem 1}, we prove that all countable elementary amenable groups of finite Hirsch length satisfy this inequality. On the other hand, in Example \ref{ex: counterex} we apply Theorem C(b) to construct groups that do  not satisfy this inequality. 

We do not claim originality for the statement on  the Bredon cohomological dimension of $\G_s$  obtained in Theorem C(a) in the case of $s=1$. This result follows from a theorem of Mart\'{\i}nez-P\'{e}rez  (see \cite[3.3]{Martinez12}), which led us to the construction of the groups $\G_s$. However, we note that our methods for computing the  Bredon cohomological dimension of these groups differ from  those of \cite{Martinez12}.

In many computations, we use a Mayer-Vietoris long exact sequence of L\"{u}ck and Weiermann which can be constructed using a certain push-out of classifying spaces. In the last section, we give a purely algebraic proof of the existence of this sequence which has the added benefit that it can be generalized to Tor and Ext functors.

\section{Preliminaries and Notations}
Let $G$ be a discrete group and let $\mathcal{F}$ be a \emph{family} of subgroups of $G$, i.e.~a collection of subgroups of $G$ that is closed under conjugation and taking subgroups. A \emph{classifying space of $G$} for the family $\mathcal{F}$ is a $G$-CW-complex $X$ such that
 $X^H=\emptyset$ when $H \notin \mathcal{F}$ and $X^H$ is contractible for every $H \in \mathcal{F}$.
A classifying space of $G$ for the family $\mathcal{F}$ is also called a \emph{model for $E_{\mathcal{F}}G$}. It can be shown such a model always exists
and  is unique up to $G$-homotopy. The smallest possible dimension of a model for $E_{\mathcal{F}}G$  is called the \emph{geometric dimension} of $G$ for the family $\mathcal{F}$ and is denoted by $\mathrm{gd}_{\mathcal{F}}(G)$. When a finite dimensional model does not exist, then $\mathrm{gd}_{\mathcal{F}}(G)$ is said to be infinite.  When $\mathcal{F}$ is the family of finite (virtually cyclic) subgroups, then $E_{\mathcal{F}}G$ and $\mathrm{gd}_{\mathcal{F}}G$ are denoted by $\underline{E}G$ ($\underline{\underline{E}}G$) and $\underline{\mathrm{gd}}G$ ($\underline{\underline{\mathrm{gd}}}G$), respectively.

We will now state some known results on the finiteness of $\mathrm{gd}_{\mathcal{F}}(G)$ that will be used throughout this paper, some without explicit referral.

The first result gives an upper bound for $\mathrm{gd}_{\mathcal{F}}(G)$, when $\mathcal{F}$ is either the family of subgroups of finitely generated subgroups  or the family of countable subgroups of $G$.
\begin{theorem} \label{th: fin gen subgroups} Let $G$ be a group of cardinality $\aleph_n$, for some natural number $n$, let $\mathcal{F}$ be the family of subgroups of finitely generated subgroups of $G$ and let $\mathcal{H}$ be the family of countable subgroups of $G$. Then there exists a model for $E_{\mathcal{F}}G$ of dimension $n+1$, and a model for $E_{\mathcal{H}}G$ of dimension $n$.
\end{theorem}
\begin{proof} The statement for the family $\mathcal{F}$ is exactly Theorem 5.31 in {L\"{u}ck-Weiermann, \cite[5.31]{LuckWeiermann}}. The proof of this theorem can be easily adapted to deal with the family of countable subgroups $\mathcal{H}$.
\end{proof}

Let $G$ be a group and let $\mathcal{F}$ and $\mathcal{H}$ be two families of subgroups of $G$ such that $\mathcal{F} \subseteq \mathcal{H}$. Next, we will deal with the following questions.
\smallskip
\begin{itemize}
\item[(a)] How can a model for $E_{\mathcal{F}}G$ be obtained from a model for $E_{\mathcal{H}}G$?
\medskip
\item[(b)] How can a model for $E_{\mathcal{H}}G$ be obtained from a model for $E_{\mathcal{F}}G$?
\end{itemize}
\smallskip

For a subgroup $H$ of  $G$, one can define the family
\[ \mathcal{F} \cap H= \{ H \cap F \ | \ F \in \mathcal{F}\} \]
of subgroups of $H$. By restricting the action of $G$ to $H$, any model for $E_{\mathcal{F}}G$ becomes a model  for $E_{\mathcal{F}\cap H}H$. This implies that $\mathrm{gd}_{\mathcal{F} \cap H}(H) \leq \mathrm{gd}_{\mathcal{F}}(G)$.

The following theorem gives us an answer to (a).
\begin{theorem}[{L\"{u}ck-Weiermann, \cite[5.1]{LuckWeiermann}}] \label{th: top spec seq} Let $G$ be a group and let $\mathcal{F}$ and $\mathcal{H}$ be two families of subgroups of $G$ such that $\mathcal{F} \subseteq \mathcal{H}$.
If there is an integer $k \geq 0$ such that for each $H \in \mathcal{H}$, there is a $k$-dimensional model for $E_{\mathcal{F}\cap H}H$. Then there exists a model for $E_{\mathcal{F}}G$ of dimension $k + \mathrm{gd}_{\mathcal{H}}(G)$.
\end{theorem}
\noindent Note that the theorem implies that $\underline{\mathrm{gd}}(G)\leq \underline{\underline{\mathrm{gd}}}(G)+1$, for any group $G$.

The next corollaries will be  key tools in our future arguments.
\begin{corollary} \label{cor: gd ext bounds} Consider a short exact sequence of groups $N \rightarrow G \xrightarrow{\pi} Q$.
Let $\mathcal{F}$ be a family of subgroups of $G$ and let $\mathcal{H}$ be a family of subgroups of $Q$ such that $\pi(\mathcal{F})\subseteq \mathcal{H}$. If there is an integer $k \geq 0$ such that for each $H \in \mathcal{H}$, there is a $k$-dimensional model for $E_{\mathcal{F}\cap \pi^{-1}(H)}\pi^{-1}(H)$. Then there exists a model for $E_{\mathcal{F}}G$ of dimension $k + \mathrm{gd}_{\mathcal{H}}(Q)$.
\end{corollary}
\begin{proof}
Define the family
\[ \mathcal{C}=\{ K \subseteq G \ | \ K \subseteq \pi^{-1}(F) \ \mbox{for some} \ F \in \mathcal{H}\} \]
of subgroups of $G$. Note that $\mathcal{F} \subseteq \mathcal{C}$. Let $X$ be a model for $E_{\mathcal{H}}Q$. By letting $G$ act on $X$ via $\pi$, it follows that $X$ is also a model for $E_{\mathcal{C}}G$. This implies that $\mathrm{gd}_{\mathcal{C}}(G) \leq \mathrm{gd}_{\mathcal{H}}(Q)$. The statement now follows from Theorem \ref{th: top spec seq}.
\end{proof}
\begin{corollary}\label{cor: useful combination} Let $G$ be a group of cardinality $\aleph_n$ and let $\mathcal{F}$ be a family of subgroups.

\begin{itemize}
\item[(i)] If there is an integer $d \geq 0$ such that for every finitely generated subgroup $H$ of $G$, there is a $d$-dimensional model for $E_{\mathcal{F}\cap H}H$. Then there exists a model for $E_{\mathcal{F}}G$ of dimension $d+n+1$.
\medskip
\item[(ii)] If there is an integer $d \geq 0$ such that for every countable subgroup $H$ of $G$, there is a $d$-dimensional model for $E_{\mathcal{F}\cap H}H$. Then there exists a model for $E_{\mathcal{F}}G$ of dimension $d+n$.
\end{itemize}
\end{corollary}
\begin{proof} This follows from combining Theorem \ref{th: fin gen subgroups} and Theorem \ref{th: top spec seq}.
\end{proof}
One way of dealing with the question (b) is  to start with a model for $E_{\mathcal{F}}G$ and then try to adapt this model to obtain a model for $E_{\mathcal{H}}G$. In section $2$ of \cite{LuckWeiermann}, L\"{u}ck and Weiermann give a general construction that follows this principle. We recall some of the basics of this construction that will be useful in the subsequent sections.

 Let $G$ be a discrete group and let $\mathcal{F}$ and $\mathcal{H}$ be families of subgroups of $G$ such that $\mathcal{F} \subseteq \mathcal{H}$ and such that there exists an equivalence relation $\sim$ on the set $\mathcal{S}=\mathcal{H}\smallsetminus \mathcal{F}$ that satisfies the following properties
\smallskip
\begin{itemize}
\item[-]  $\forall H,K \in \mathcal{S} : H \subseteq K \Rightarrow H \sim K$;
\medskip
\item[-] $ \forall H,K \in \mathcal{S},\forall x \in G: H \sim K \Leftrightarrow H^x \sim K^x$.
\end{itemize}
\smallskip
An equivalence relation that satisfies these properties is called a \emph{strong equivalence relation}.
Let $[H]$ be an equivalence class represented by $H \in \mathcal{S}$ and denote the set of equivalence classes by $[\mathcal{S}]$. The group $G$ acts on $[\mathcal{S}]$ via conjugation, and the stabilizer group of an equivalence class $[H]$ is
\begin{equation}\label{eq gen normalizer} \mathrm{N}_{G}[H]=\{x \in \Gamma \ | \ H^x \sim H \}. \end{equation}
Note that $\mathrm{N}_{G}[H]$ contains $H$ as a subgroup.
Define for each $[H] \in \mathcal{I}$ the  family
\begin{equation} \label{eq: special family}\mathcal{F}[H]=\{ K \subset \mathrm{N}_{G}[H] \ | K \in \mathcal{S}, K \sim H\} \cup \Big(\mathrm{N}_{G}[H] \cap \mathcal{F}\Big) \end{equation} of subgroups of $\mathrm{N}_{G}[H]$.
\begin{theorem}[{L\"{u}ck-Weiermann, \cite[2.5]{LuckWeiermann}}] \label{th: push out} Let $\mathcal{F} \subseteq \mathcal{H}$ be two families of subgroups of a group $G$ such that $S=\mathcal{H}\smallsetminus \mathcal{F}$ is equipped with a strong equivalence relation. Denote the set of equivalence classes by $[\mathcal{S}]$ and let $\mathcal{I}$ be a complete set of representatives $[H]$ of the orbits of the conjugation action of $G$ on $[\mathcal{S}]$. If there exists a natural number $k$ such that  $\mathrm{gd}_{\mathcal{F}\cap \mathrm{N}_{G}[H]}(\mathrm{N}_{G}[H]) \leq k-1$ and
$\mathrm{gd}_{\mathcal{F}[H]}(\mathrm{N}_{G}[H]) \leq k$ for each $[H] \in \mathcal{I}$,
and such that $\mathrm{gd}_{\mathcal{F}}(G) \leq k$, then $\mathrm{gd}_{\mathcal{H}}(G) \leq k$.
\end{theorem}
L\"{u}ck and Weiermann applied this construction to  prove the following result which will be useful for our purposes.
\begin{theorem}[{L\"{u}ck, \cite[5.26]{Luck2}} and {L\"{u}ck-Weiermann, \cite[5.13]{LuckWeiermann}}] Let $\Gamma$ be a virtually polycyclic group of Hirsch length $h$, then
$\underline{\mathrm{gd}}(\Gamma)=h$ and $\underline{\underline{\mathrm{gd}}}(\Gamma)\leq h+1$.
\end{theorem}

An important algebraic tool to study finiteness properties of classifying spaces for families of subgroups is Bredon cohomology.
Bredon cohomology was introduced by Bredon in \cite{Bredon} for finite groups and was generalized to arbitrary groups by L\"{u}ck (see \cite{Luck}).

Let $G$ be a discrete group and let $\mathcal{F}$ be a family of subgroups of $G$. The \emph{orbit category} $\orb$ is a category defined by the objects that are the left cosets $G/H$ for all $H \in \mathcal{F}$ and the morphisms are all $G$-equivariant maps between the objects.
An \emph{$\orb$-module} is a contravariant functor $M: \orb \rightarrow \mathbb{Z}\mbox{-mod}$. The \emph{category of $\orb$-modules} is denoted by $\orbmod$ and is defined by the objects that are all the $\orb$-modules and the morphisms are all the natural transformations between the objects.
The category $\orbmod$ contains enough projective and injective objects to construct projective and injective resolutions, so one can construct bi-functors $\mathrm{Ext}^{n}_{\orb}(-,-)$ that have all the usual properties. The \emph{$n$-th Bredon cohomology of $G$} with coefficients $M \in \orbmod$ is by definition
\[ \mathrm{H}^n_{\mathcal{F}}(G,M)= \mathrm{Ext}^{n}_{\orb}(\underline{\mathbb{Z}},M), \]
where $\underline{\mathbb{Z}}$ is the constant functor. There is also a notion of \emph{Bredon cohomological dimension} of $G$ for the family $\mathcal{F}$, denoted by $\mathrm{cd}_{\mathcal{F}}(G)$ and defined by
\[ \mathrm{cd}_{\mathcal{F}}(G) = \sup\{ n \in \mathbb{N} \ | \ \exists M \in \orbmod :  \mathrm{H}^n_{\mathcal{F}}(G,M)\neq 0 \}. \]
When $\mathcal{F}$ is the family of finite (virtually cyclic) subgroups, then $\mathrm{cd}_{\mathcal{F}}G$ is denoted by $\underline{\mathrm{cd}}G$ ($\underline{\underline{\mathrm{cd}}}G$).
Since the augmented cellular chain complex of any model for $E_{\mF}G$ yields a projective resolution of $\underline{\mathbb{Z}}$ that can be used to compute $\mathrm{H}_{\mF}^{\ast}(G,-)$, it follows that $ \mathrm{cd}_{\mathcal{F}}(G) \leq  \mathrm{gd}_{\mathcal{F}}(G)$. In fact, in \cite[0.1]{LuckMeintrup}, L\"{u}ck and Meintrup even showed that  \[\mathrm{cd}_{\mathcal{F}}(G) \leq  \mathrm{gd}_{\mathcal{F}}(G) \leq \max\{3, \mathrm{cd}_{\mathcal{F}}(G) \}.\]
Also note that for any group $G$, one has $\mathrm{cd}_{\mathbb{Q}}(G) \leq \underline{\mathrm{cd}}(G)$. 

In the last section, we will briefly discuss some of the other technical aspects of Bredon cohomology. We will also recall the definition of Bredon homology and Bredon homological dimension, as we will need these notions in Section 6. We refer the reader to \cite{Luck} and \cite{FluchThesis} for a detailed introduction to Bredon (co)homology.
\begin{notation} \rm
\begin{itemize} \item[]
\item[-] If $a_1,a_2,\ldots,a_n$ are elements of some group $G$, then $\prod_{i=1}^na_i$ denotes the product $a_1a_2\ldots a_n$ in that specific order.
\smallskip
\item[-] By $\mathbb{N}_1$ we denote the set of natural numbers without the number zero.
\smallskip
\item[-] If $t$ is an element of some group $G$, then $\langle t \rangle$ is the cyclic subgroup of $G$ generated by $t$. All infinite cyclic groups will be written multiplicatively with unit $1$. The unit element of a group that is not infinite cyclic will be denoted by $e$.
\smallskip
\item[-] If $G$ is a group and $C=\langle t \rangle$ is an infinite cyclic group generated by $t$, we say that a semi-direct product $G\rtimes C$ is determined by $\varphi \in \mathrm{Aut}(G)$ when
\[ (e,t)(g,1)(e,t^{-1})=(\varphi(g),1) \in G \rtimes C\]
for all $g \in G$ and denote this group by $G\rtimes_{\varphi} C$.
\smallskip
\item[-] For any $d\in \mathbb{N}_1$, we will denote by $C_d$ the cyclic group of order $d$.
\smallskip
\item[-] Unless stated otherwise, $\mathcal{F}$ will always denote the family of finite subgroups of a given group.
\end{itemize}
\end{notation}

\section{Commensurators of infinite virtually cyclic groups}
In this section, we apply the construction of L\"{u}ck and Weiermann when $\mathcal F$  and $\mathcal H$ are the families of finite and virtually cyclic subgroups of $G$, respectively. This naturally leads us to consider commensurators of infinite virtually cyclic subgroups of $G$. We derive some general facts about such commensurators that will be essential to the proof of Theorem A. 

The commensurator of a subgroup $H$ of a group $G$ is by definition the group
\[ \mathrm{Comm}_{G}[H]=\{x \in G \ | \ [H: H \cap H^{x}]<\infty \ \mbox{and} \ [H^x: H \cap H^{x}]< \infty \}. \]
When $H$ is an infinite virtually cyclic subgroup of $G$ then one can verify that
\[ \mathrm{Comm}_{G}[H]=\{x \in G \ | \ |H \cap H^x | = \infty \}. \]

Let us denote the set of infinite virtually cyclic subgroups of $G$ by $\mathcal{S}$. As in definition $2.2.$ of \cite{LuckWeiermann}, two infinite virtually cyclic subgroups $H$ and $K$ of $\Gamma$ are said to be \emph{equivalent}, denoted $H \sim K$, if $|H\cap K|=\infty$. Using the fact that any two infinite virtually cyclic subgroups of a virtually cyclic group are equivalent (e.g. see Lemma 3.1. in \cite{DP}), it is easily seen that this indeed defines an equivalence relation on $\mathcal{S}$. One can also verify that this equivalence relation is a strong equivalence relation. Note that, in this context, the group defined in (\ref{eq gen normalizer}) is exactly the commensurator of $H$ in the group $G$. Moreover, the family $\mathcal{F}[H]$ defined in (\ref{eq: special family}) consists of all finite subgroups of $\mathrm{Comm}_{G}[H]$ and all infinite virtually cyclic subgroups of $\mathrm{Comm}_{G}[H]$ that are equivalent to $H$. From Theorem \ref{th: push out}, it now follows that in order to determine an upper bound for $\underline{\underline{\mathrm{gd}}}(G)$, we need find an upper bound for $\mathrm{gd}_{\mathcal{F}[H]}(\mathrm{Comm}_{G}[H])$. This calls for a more careful study of the group $\mathrm{Comm}_{G}[H]$, which is the purpose of this section.

We begin by pointing out some basic properties of commensurators:
\begin{itemize}
\item[(i)] If $H$ and $K$ are two equivalent infinite virtually cyclic subgroups of a group $G$, then their commensurators coincide. In particular, the commensurator of an infinite virtually cyclic subgroup $H$ of $G$ is the commensurator of any infinite cyclic subgroup of $H$.
\item[(ii)]  If $G$ is a subgroup of $\Gamma$ such that $H \subseteq G$, then $\mathrm{Comm}_{G}[H]= \mathrm{Comm}_{\Gamma}[H]\cap G$. Also, if $\pi: G \rightarrow Q$ is a group homomorphism such that $|\pi(H)|=\infty$, then $\pi(\mathrm{Comm}_{G}[H])\subseteq \mathrm{Comm}_{Q}[\pi(H)]$.
\end{itemize}

In the next results, we investigate the structure of commensurators of infinite virtually cyclic groups inside semi-direct products $N\rtimes \mathbb{Z}$.
\begin{lemma} \label{lemma: key lemma}Let $N$ be a group and consider $\Gamma=N \rtimes \mathbb{Z}$. Let $H$ be an infinite virtually cyclic subgroup of
$\Gamma$ that is not contained in $N$. Then there is a semi-direct product
$\mathrm{Comm}_{\Gamma}[H]=N_0 \rtimes_{\varphi} \langle t \rangle$, with $N_0$
a subgroup of $N$. Moreover, there exists an element $a \in N_0$ and an integer $n \in \mathbb{N}_1$ such that $\langle(a,t^n) \rangle$ is an infinite cyclic subgroup of $H$ and
\smallskip
\begin{itemize}
\item[(i)] $\varphi(a)=a$;
\medskip
\item[(ii)] for each $x \in N_0$ there exists $m \in \mathbb{N}_1$ such that $\varphi^{nm}(x)=a^{-m}xa^m$.
\end{itemize}
\end{lemma}
\begin{proof}
We may assume that $H$ is infinite cyclic.
Since $H$ is contained in $\mathrm{Comm}_{\Gamma}[H]$ but not contained
in $N$, we have a split short exact sequence,
\[ 1 \rightarrow N_0 \rightarrow \mathrm{Comm}_{\Gamma}[H] \rightarrow
\mathbb{Z}=\langle t \rangle \rightarrow 0 \]
determined by some $\varphi \in \mathrm{Aut}(N_0)$, where $N_0$ is a
subgroup of $N$. Let $(b,t^r)$ be a generator of $H$, with $r \in \mathbb{N}_1$. Since $(e,t) \in
\mathrm{Comm}_{\Gamma}[H]$, we must have that $\langle
{}^{(e,t)}(b,t^r)\rangle\sim \langle(b,t^r) \rangle$. This implies that
there exists a $k \in \mathbb{N}_1$ such that \[
(e,t)(b,t^r)^k(e,t^{-1})=(b,t^r)^k.\]
Denoting $a=\prod_{i=0}^{k-1}\varphi^{ir}(b)$,  the equation above
implies that $\varphi(a)=a$. Let $n=rk$, then $\langle (a,t^n) \rangle$ is an infinite cyclic subgroup of $H$ such that $\varphi(a)=a$.
Note that $(a,t^n)^s=(a^s,t^{ns})$ for each $s \in \mathbb{N}$.
Let $x \in N_0$. Since $(x,1) \in \mathrm{Comm}_{\Gamma}[H]$ , we know
there exists  an $m \in \mathbb{N}_1$ such that \[
(x,1)(a,t^n)^m(x^{-1},1)=(a,t^n)^m.\]
This implies that $(xa^m\varphi^{nm}(x)^{-1},t^{nm})=(a^m,t^{nm})$, and
hence  $\varphi^{nm}(x)=a^{-m}xa^m$.
\end{proof}
\begin{proposition} \label{prop: finitely generated} Let $N$ be a group
and consider $\Gamma=N \rtimes \mathbb{Z}$.
Let $H$ be an infinite virtually cyclic subgroup of $\Gamma$ that is not
contained in $N$. Then the following hold.
\smallskip
\begin{itemize}
\item[(a)] Every finitely generated subgroup of
$\mathrm{Comm}_{\Gamma}[H]$ that contains $H$, is a
semi-direct product of a finitely generated subgroup of $N$ with
$\mathbb{Z}$.
\medskip
\item[(b)] Let $K$ be a finitely generated subgroup of
$\mathrm{Comm}_{\Gamma}[H]$ that contains $H$. Then $H$ has a finite index subgroup contained in the center of $K$.
\end{itemize}
\end{proposition}
\begin{proof}
Let $K$ be a finitely generated subgroup of $\mathrm{Comm}_{\Gamma}[H]$ that contains $H$. By Lemma \ref{lemma: key lemma}, we have that
$K=N_0 \rtimes_{\varphi} \langle t \rangle$, where $N_0$
a subgroup of $N$. Also by Lemma \ref{lemma: key lemma}, there exists an element $a \in N_0$ and an integer $n \in \mathbb{N}_1$ such that $\langle(a,t^n) \rangle$ is an infinite cyclic subgroup of $H$ and for which $\varphi(a)=a$. Moreover, for each $x \in N_0$ there exists $m \in \mathbb{N}_1$ such that $\varphi^{nm}(x)=a^{-m}xa^m$. Since $K$ is finitely generated, it can be generated by elements of the form
\[ (x_1,1),(x_2,1),\ldots,(x_k,1),(e,t). \]
For each  $i \in \{1,\ldots,k\}$, take  $l_i \in \mathbb{N}_1$ such that
$\varphi^{nl_i}(x_i)=a^{-l_i}x_ia^{l_i}$. Denote $l=\prod_{i=1}^k l_i$. Then
$\varphi^{ln}(x_i)=a^{-l}x_ia^{l}$ for all $i \in \{1,\ldots,k\}$.
Now, the set
\[ \Big\{\varphi^j(x_i) \ | \ i \in \{1,\ldots,k\}, j \in
\{0,\ldots,ln-1\}\Big\} \cup \{a\} \]
is a set of generators of $N_0$ and $\langle (a^l,t^{nl})\rangle = \langle (a,t^n)^l\rangle$ is a finite index subgroup of $H$ that is contained in the center of $K$.
\end{proof}
The previous proposition allows us to determine the structure of commensurators of infinite virtually cyclic groups inside certain elementary amenable groups.
\begin{proposition} \label{prop: loc fin by virt poly}
Let $\Lambda$ be a locally finite group, $N$ be a torsion-free
nilpotent group, and $$\Lambda \rightarrow \overline{N} \rightarrow N$$ be a short exact sequence of groups.
Let $Q$ be a virtually cyclic group and suppose that there is also a short exact sequence of groups
$$\overline{N} \rightarrow G \xrightarrow{\pi} Q.$$
Let $H$ be an infinite virtually cyclic subgroup of $G$ such that $\mathrm{Comm}_G[H]=G$ and such that $\pi(H)$ is a finite index subgroup of $Q$.
Then every finitely generated subgroup of $G$ is countable locally
finite-by-virtually polycyclic.
\end{proposition}
\begin{proof}
We may assume that $H$ is infinite cyclic.
First consider the case where $\pi(H)$, and hence $Q$, is finite. Let $K$ be a
finitely generated subgroup of $G$. Then $P=K\cap  \overline{N}$ is a
finite index subgroup of $K$ and therefore also finitely generated. It
follows that $P$ is countable locally finite-by-poly-$\mathbb{Z}$.
Hence, $K$ is countable locally finite-by-virtually polycyclic. \\
\indent Now assume that $\pi(H)$ is infinite. Let $K$ be a finitely generated
subgroup of $G$ that contains $H$. Then $P = K \cap
\pi^{-1}(\pi(H))$ is a finite index subgroup of $K$ that contains $H$
and it is finitely generated. Since $\pi(H)$ is infinite, we have $\pi^{-1}(\pi(H))\cong \overline{N}\rtimes \mathbb{Z}$.
This implies that $P$ is isomorphic to a subgroup of a semi-direct product $\overline{N}\rtimes \mathbb{Z}$
such that $H$ is not contained in $\overline{N}$ and
$\mathrm{Comm}_{\overline{N}\rtimes \mathbb{Z}}[H]=\overline{N}\rtimes
\mathbb{Z}$. It follows from Proposition \ref{prop: finitely
generated}(a) that $P$ is countable locally finite-by-poly-$\mathbb{Z}$. We conclude that $K$ is countable locally
finite-by-virtually polycyclic. This ends the proof, since every finitely generated subgroup of $G$ is contained in a finitely generated subgroup of $G$ that contains $H$.
\end{proof}

\section{Elementary amenable groups}
The class of elementary amenable groups is by definition the smallest class of groups that is closed under taking directed unions, extensions, subgroups and quotients, and that contains all finite groups and all abelian groups. This class of groups contains, for example, all virtually solvable groups. In \cite{Hillman}, Hillman extended the notion of Hirsch length from polycyclic groups to elementary amenable groups.
Its behaviour under taking subgroups, directed unions, and extensions is similar to that of for polycyclic groups. We refer the reader to section $1$ in \cite{Hillman} for the precise definition of Hirsch length.

We shall only consider elementary amenable groups of finite Hirsch length. For these groups one has the following structure theorem.
\begin{theorem}[{Wehrfritz, \cite[(g)]{Wehrfritz}, see also Hillman-Linnell, \cite{HillmannLinnel}}] \label{th: struc elem am} Let $\Gamma$ be an elementary amenable group of finite Hirsch length $h$.
There exist characteristic subgroups $\Lambda(\Gamma)\subseteq N\subseteq M$ of $\Gamma$ and a function $j: \mathbb{N} \rightarrow \mathbb{N}$ such that
\smallskip
\begin{itemize}
\item[(a)] $\Lambda(\Gamma)$ is the unique maximal normal locally finite subgroup of $\Gamma$, i.e. the locally finite radical;
\medskip
\item[(b)] $N/\Lambda(\Gamma)$ is torsion-free nilpotent;
\medskip
\item[(c)] $M/N$ is free abelian of finite rank;
\medskip
\item[(d)] $[\Gamma:M] \leq j(h)$.
\end{itemize}
\end{theorem}
Our task in this section is to prove Theorem A. But first, we need a few lemmas. We begin with a result that is closely related to Lemma 5.15 of \cite{LuckWeiermann}.
\begin{lemma} \label{lemma: virt poly} Let $H$ be a infinite cyclic
subgroup of a group $\Gamma$. If an infinite cyclic subgroup $K$ of $H$
is normal in $\mathrm{Comm}_{\Gamma}[H]$, then
$\mathrm{gd}_{\mathcal{F}[H]}(\mathrm{Comm}_{\Gamma}[H])=
\underline{\mathrm{gd}}(\mathrm{Comm}_{\Gamma}[H]/K)$ and $\mathrm{cd}_{\mathcal{F}[H]}(\mathrm{Comm}_{\Gamma}[H])=
\underline{\mathrm{cd}}(\mathrm{Comm}_{\Gamma}[H]/K)$. If $\mathrm{Comm}_{\Gamma}[H]$ is virtually polycyclic, then $H$ always has an infinite cyclic subgroup that is normal in $\mathrm{Comm}_{\Gamma}[H]$.
\end{lemma}
\begin{proof}
We may assume that $H=K$ and that $H$ is normal in
$\mathrm{Comm}_{\Gamma}[H]$. Consider the extension
\begin{equation}\label{eq: comm ext} 0 \rightarrow  H \rightarrow \mathrm{Comm}_{\Gamma}[H]
\xrightarrow{\pi} \mathrm{Comm}_{\Gamma}[H]/ H \rightarrow 1 . \end{equation}
Note that every group in $\mathcal{F}[H]$ is mapped onto a finite
subgroup of $\mathrm{Comm}_{\Gamma}[H]/ H$ by $\pi$. On the other hand,
$\pi^{-1}(F) \in \mathcal{F}[H] $ for every finite subgroup $F$ of
$\mathrm{Comm}_{\Gamma}[H]/ H$. Therefore, it follows from Corollary
\ref{cor: gd ext bounds} that
$\mathrm{gd}_{\mathcal{F}[H]}(\mathrm{Comm}_{\Gamma}[H])\leq
\underline{\mathrm{gd}}(\mathrm{Comm}_{\Gamma}[H]/H)$.
On the other hand, if $X$ is a model for $E_{\mathcal{F}[H]}\mathrm{Comm}_{\Gamma}[H]$ then it is not difficult to verify that $X^H$ will be a model for $\underline{\mathrm{gd}}(\mathrm{Comm}_{\Gamma}[H]/H)$. We conclude that $\mathrm{gd}_{\mathcal{F}[H]}(\mathrm{Comm}_{\Gamma}[H])=
\underline{\mathrm{gd}}(\mathrm{Comm}_{\Gamma}[H]/H)$. To see that $\mathrm{cd}_{\mathcal{F}[H]}(\mathrm{Comm}_{\Gamma}[H])=
\underline{\mathrm{cd}}(\mathrm{Comm}_{\Gamma}[H]/H)$, let $\mathcal{H}$ be the family of finite subgroups of $ \mathrm{Comm}_{\Gamma}[H]/ H$ and consider the spectral sequence 
\[ E_2^{p,q}(M)=\mathrm{H}^p_{\mathcal{H}}(\mathrm{Comm}_{\Gamma}[H]/H, \mathrm{H}^{q}_{\mathcal{F}[H] \cap \pi^{-1}(-)}(\pi^{-1}(-),M)) \Longrightarrow \mathrm{H}^{p+q}_{\mathcal{F}[H]}(\mathrm{Comm}_{\Gamma}[H],M)                \]
for $M \in \mbox{Mod-}\mathcal{O}_{\mathcal{F}[H]}\mathrm{Comm}_{\Gamma}[H]$, associated to extension (\ref{eq: comm ext}) (see \cite[Th. 5.1]{Martinez}).
Since $\pi^{-1}(F) \in \mathcal{F}[H] $ for every finite subgroup $F$ of $\mathrm{Comm}_{\Gamma}[H]/ H$, we have $E_2^{p,q}(M)=0$ for all $q>1$ and $E_2^{n,0}(M)=\mathrm{H}^n_{\mathcal{H}}(\mathrm{Comm}_{\Gamma}[H]/H,M(\mathrm{Comm}_{\Gamma}[H]/\pi^{-1}(-)))$ . This implies that \[ \mathrm{H}^{n}_{\mathcal{F}[H]}(\mathrm{Comm}_{\Gamma}[H],M)= \mathrm{H}^n_{\mathcal{H}}(\mathrm{Comm}_{\Gamma}[H]/H,M(\mathrm{Comm}_{\Gamma}[H]/\pi^{-1}(-)))\] for every $M \in \mbox{Mod-}\mathcal{O}_{\mathcal{F}[H]}\mathrm{Comm}_{\Gamma}[H]$. Let $L \in  \mbox{Mod-}\mathcal{O}_{\mathcal{H}}\mathrm{Comm}_{\Gamma}[H]/H$. Then we have $L \circ \pi \in  \mbox{Mod-}_{\mathcal{O}_{\mathcal{F}}[H]}\mathrm{Comm}_{\Gamma}[H]$ and \[\mathrm{H}^{n}_{\mathcal{F}[H]}(\mathrm{Comm}_{\Gamma}[H],L \circ \pi)= \mathrm{H}^n_{\mathcal{H}}(\mathrm{Comm}_{\Gamma}[H]/H,L)\] for every $L \in  \mbox{Mod-}\mathcal{O}_{\mathcal{H}}\mathrm{Comm}_{\Gamma}[H]/H$. Therefore, we may conclude that \[\mathrm{cd}_{\mathcal{F}[H]}(\mathrm{Comm}_{\Gamma}[H])=
\underline{\mathrm{cd}}(\mathrm{Comm}_{\Gamma}[H]/H).\]

Now, let us assume that $\mathrm{Comm}_{\Gamma}[H]$ is virtually polycyclic. Let
$x_1,\ldots,x_r$ be a set of generators for $\mathrm{Comm}_{\Gamma}[H]$
and let $t$ be a generator for $H$. It follows that for every $i \in
\{1,\ldots,r\}$, we have $m_i, n_i \in \mathbb{Z} \smallsetminus \{0\}$ such
that $x_i^{-1}t^{n_i}x_i=t^{m_i}$. Lemma $5.14$(v) in
\cite{LuckWeiermann} implies that $m_i=\pm n_i$ for every  $i \in
\{1,\ldots,r\}$. Define $n= \prod_{i=1}^r n_i$. One checks that $x_i^{-1}t^{n}x_i=t^{\pm n}$ for all $i
\in \{1,\ldots,r\}$. This shows that $\langle t^n \rangle$ is a normal
subgroup of $\mathrm{Comm}_{\Gamma}[H]$, and finishes the proof.
\end{proof}
\begin{lemma} \label{lemma: locally finite by virtually cyclic} Let $H$
be an infinite cyclic subgroup of a group $\Gamma$. If
$\mathrm{Comm}_{\Gamma}[H]$ is countable locally finite-by-virtually
cyclic, then $\mathrm{Comm}_{\Gamma}[H]$ is locally virtually cyclic and $\mathcal{F}[H]$ is the family of virtually cyclic subgroups of $\mathrm{Comm}_{\Gamma}[H]$. In particular, we have that $\mathrm{gd}_{\mathcal{F}[H]}(\mathrm{Comm}_{\Gamma}[H])\leq 1$.
\end{lemma}
\begin{proof}
By assumption, there is an extension
\[ 1 \rightarrow L \rightarrow \mathrm{Comm}_{\Gamma}[H]
\xrightarrow{\pi}  V \rightarrow 1,\]
where $L$ is countable locally finite and $V$ is virtually cyclic. Note that
$\pi$ maps $H$ isomorphically onto a infinite cyclic subgroup $S$ of
$V$. Hence, $\pi^{-1}(S)$ is a finite index subgroup of
$\mathrm{Comm}_{\Gamma}[H]$ that is isomorphic to a semi-direct product
$L \rtimes H$. It follows from Proposition $3.4.$(b) of \cite{DP} that the
family $\mathcal{F}[H]\cap \pi^{-1}(S)$ coincides with the family of virtually cyclic
subgroups of $\pi^{-1}(S)$. Hence, $\mathcal{F}[H]$ is the family of virtually cyclic subgroups of $\mathrm{Comm}_{\Gamma}[H]$.
The remaining statements now follow from considering the finite index subgroup $\pi^{-1}(S)$ of $\mathrm{Comm}_{\Gamma}[H]$, Proposition $3.4.$(a) and Proposition $4.3.$ in \cite{DP}
\end{proof}
We are ready to prove Theorem A.
\begin{proof}[Proof of Theorem A]
First note that the general case follows from the countable case and Corollary \ref{cor: useful combination}. Let $\Gamma$ be a countable elementary amenable group of finite Hirsch length $h$. If $h=0$ then $\Gamma$ is locally finite, in which case $\underline{\underline{\mathrm{gd}}}(\Gamma)=\underline{\mathrm{gd}}(\Gamma)\leq 1$ by Theorem \ref{th: fin gen subgroups}. Now, assume that $h\geq 1$. Let $\mathcal{S}$ be the set of infinite virtually cyclic subgroups of $\Gamma$ and equip $\mathcal{S}$ with the equivalence relation discussed in Section 2. Let $[H]$ be the equivalence class represented by $H \in \mathcal{S}$ and denote the set of equivalence classes by $[\mathcal{S}]$. Let $\mathcal{I}$ be a complete set of representatives $[H]$ of the $\Gamma$-orbits of $[\mathcal{S}]$ induced by conjugation. In \cite{FloresNuc}, Flores and Nucinkis proved that $\underline{\mathrm{cd}}(\Gamma) \leq h+1$. If $h=1$, this implies that $\underline{\mathrm{gd}}(\Gamma) \leq 3=h+2$. If $h \geq 2$ however, we can conclude that
$\underline{\mathrm{gd}}(\Gamma) \leq h+1$. Because of this discrepancy, we need to distinguish between  two cases.

\textbf{Case 1: $\mathbf{h=1}$.} We claim that there exists a model for $E_{\mathcal{F}[H]}(\mathrm{Comm}_{\Gamma}[H])$ of dimension $1$ and there is a model for $\underline{E}(\mathrm{Comm}_{\Gamma}[H])$ of dimension $2$, for each $[H] \in \mathcal{I}$.
To this end, consider $[H] \in \mathcal{I}$. Let $K$ be a finitely generated subgroup of $\mathrm{Comm}_{\Gamma}[H]$ that contains $H$. It follows from Theorem \ref{th: struc elem am} that $K$ is countable locally finite-by-virtually cyclic. Lemma \ref{lemma: locally finite by virtually cyclic} now implies that $K$ is virtually cyclic. Hence, $\mathrm{Comm}_{\Gamma}[H]$ is locally virtually cyclic and the family $\mathcal{F}[H]$ coincides with the family of virtually cyclic subgroups of $\mathrm{Comm}_{\Gamma}[H]$.
Applying Corollary \ref{cor: useful combination}, our claims follow. Theorem \ref{th: push out} now implies that $\underline{\underline{\mathrm{gd}}}(\Gamma)\leq 3=h+2$.

\textbf{Case 2: $\mathbf{h\geq 2}$.} We claim that in this case, there exists a model for $E_{\mathcal{F}[H]}(\mathrm{Comm}_{\Gamma}[H])$ of dimension $h+1$, for each $[H] \in \mathcal{I}$.
To prove this, consider $[H] \in \mathcal{I}$ and denote $G=\mathrm{Comm}_{\Gamma}[H]$. It follows from Theorem \ref{th: struc elem am} that there exists a countable locally finite group $\Lambda$, a torsion-free nilpotent group $N$ and a finitely generated virtually abelian group $A$ that fit into extensions
\[ 1 \rightarrow \Lambda \rightarrow \overline{N} \rightarrow N \rightarrow 1
\;\; \mbox{ and } \;\; 1 \rightarrow \overline{N} \rightarrow G \xrightarrow{\pi} A \rightarrow 1.\]
We will now consider two separate cases.

\textbf{Case 2.a: $\mathbf{\pi(H)}$ is a finite subgroup of $\mathbf{A}$.} In this case, $H$ has an infinite cyclic subgroup that is contained in $\overline{N}$, hence we may as well assume that $H$ is infinite cyclic and contained in $\overline{N}$.
Note that every group in $\mathcal{F}[H]$ is mapped onto a finite subgroup of $A$ by $\pi$. Using Corollary \ref{cor: gd ext bounds}, we obtain
\[      \mathrm{gd}_{\mathcal{F}[H]}(G) \leq \underline{\mathrm{gd}}(A)+  \sup_{F \subseteq A, |F|<\infty}   \{\mathrm{gd}_{\mathcal{F}[H]\cap \pi^{-1}(F)}(\pi^{-1}(F))\}.                                  \]
Let $F$ be a finite subgroup of $A$ and denote $G_0=\pi^{-1}(F)$. Note that $G_0$ contains $H$, that $G_0=\mathrm{Comm}_{G_0}[H]$ and that $\mathcal{F}[H]\cap G_0$ consists of the finite subgroups of $G_0$ and the infinite virtually cyclic subgroups of $G_0$ that are equivalent to $H$. Let $K$ be a finitely generated subgroup of $G_0$ that contains $H$. It follows from Proposition \ref{prop: loc fin by virt poly} that there is a short exact sequence
\[ 1 \rightarrow D \rightarrow K \xrightarrow{p}  Q \rightarrow 1, \]
with $D$ countable locally finite and $Q $ virtually polycyclic of Hirsch length at most $h(N)$.
Note that $p(H)$ is an infinite cyclic subgroup of $Q$. Let $\mathcal{H}$ be the family of subgroups of $Q$ that contains all finite subgroup of $Q$ and all infinite virtually cyclic subgroups of $Q$ that are equivalent to $p(H)$. Corollary \ref{cor: gd ext bounds} implies that
\[\mathrm{gd}_{\mathcal{F}[H]\cap K}(K)\leq \mathrm{gd}_{\mathcal{H}}Q+ \sup_{S \in \mathcal{H}}\{\mathrm{gd}_{\mathcal{F}[H]\cap p^{-1}(S)}(p^{-1}(S))\}. \]
Furthermore, it follows from Lemma \ref{lemma: virt poly} that \[\mathrm{gd}_{\mathcal{H}}Q \leq \underline{\mathrm{gd}}(Q/p(H)) \leq h(N)-1\] and by Lemma \ref{lemma: locally finite by virtually cyclic} that \[\sup_{S \in \mathcal{H}}\{\mathrm{gd}_{\mathcal{F}[H]\cap p^{-1}(S)}(p^{-1}(S))\} \leq 1.\]
We can therefore conclude that  $\mathrm{gd}_{\mathcal{F}[H]\cap K}(K)\leq h(N)$. This implies that $\mathrm{gd}_{\mathcal{F}[H]\cap G_0}(G_0)\leq h(N)+1$, since every finitely generated subgroup of $G_0$ is contained in a finitely generated subgroup of $G_0$ that contains $H$. It follows that \[\max_{F \subseteq  A, |F|<\infty}   \{\mathrm{gd}_{\mathcal{F}[H]\cap \pi^{-1}(F)}(\pi^{-1}(F))\} \leq h(N)+1. \] Since $\underline{\mathrm{gd}}(A)=h(A)$, we find that $\mathrm{gd}_{\mathcal{F}[H]}(G) \leq h+1$. \\

\textbf{Case 2.b: $\mathbf{\pi(H)}$ is an infinite subgroup of $\mathbf{A}$.} We may assume that $H$ is infinite cyclic. Denote by $\mathcal{H}$ the family of subgroups of $A$ that contains all finite subgroups of $A$ and all infinite virtually cyclic subgroups that are equivalent to $\pi(H)$. By Corollary \ref{cor: gd ext bounds},   we have
\[      \mathrm{gd}_{\mathcal{F}[H]}(G) \leq \mathrm{gd}_{\mathcal{H}}(A)+  \max_{V \in \mathcal{H}}   \{\mathrm{gd}_{\mathcal{F}[H]\cap \pi^{-1}(V)}(\pi^{-1}(V))\}.        \]
Since $\mathrm{Comm}_A[\pi(H)]=A$ and $A$ is virtually polycyclic, by Lemma \ref{lemma: virt poly}, we have \[\mathrm{gd}_{\mathcal{H}}(A)= \underline{\mathrm{gd}}(A/\pi(H))\leq h(A)-1.\]
Let $V \in \mathcal{H}$ and denote $\pi^{-1}(V)$ by $G_0$. Let $K$ be a finitely generated subgroup of $G_0$ that contains $H$. Again, by Proposition \ref{prop: loc fin by virt poly}, we have a short exact sequence
\[ 1 \rightarrow D \rightarrow K \xrightarrow{p} Q \rightarrow 1, \]
with $D$ countable locally finite and $Q$ virtually polycyclic. However, this time the Hirsch length of $Q$ is at most $h(N)+1$.
Using the same approach as before, it follows that $\mathrm{gd}_{\mathcal{F}[H]\cap K}(K)\leq h(N)+1$, and hence $\mathrm{gd}_{\mathcal{F}[H]\cap G_0}(G_0)\leq h(N)+2$.
We conclude that \[\max_{V \in \mathcal{H}}   \{\mathrm{gd}_{\mathcal{F}[H]\cap \pi^{-1}(V)}(\pi^{-1}(V))\} \leq h(N)+2\] and $\mathrm{gd}_{\mathcal{F}[H]}(G) \leq h+1$. This proves the claim.\\
\indent Since $\underline{\mathrm{gd}}(\Gamma) \leq h+1$, it follows from Theorem \ref{th: push out} that
$\underline{\underline{\mathrm{gd}}}(\Gamma)\leq h+2$.
\end{proof}
Theorem A has the following corollary.
\begin{corollary} \label{cor: cor main theorem 1} Let $\Gamma$ be a countable elementary amenable group of finite Hirsch length $h$, then
\[ \underline{\mathrm{gd}}(\Gamma)-1  \leq \underline{\underline{\mathrm{gd}}}(\Gamma)\leq \underline{\mathrm{gd}}(\Gamma)+1. \]
\end{corollary}
\begin{proof} In the proof of Theorem A, we showed that $\mathrm{gd}_{\mathcal{F}[H]}(\mathrm{Comm}_{\Gamma}[H])\leq h+1$ for any infinite virtually cyclic subgroup $H$ of $\Gamma$. In \cite{FloresNuc}, Flores and Nucinkis established the inequality $h\leq \underline{\mathrm{cd}}(\Gamma)$. Since this implies $h\leq \underline{\mathrm{gd}}(\Gamma)$, it follows from Theorem \ref{th: push out} that $\underline{\underline{\mathrm{gd}}}(\Gamma)\leq \underline{\mathrm{gd}}(\Gamma)+1$.
\end{proof}

\section{Examples}

In this section, we will compute $\underline{\underline{\mathrm{gd}}}(\G)$ for some elementary amenable groups $\G$ of finite Hirsch length. Our examples will illustrate that  the bounds of Corollary \ref{cor: cor main theorem 1} cannot be improved and all intermediate values can  be attained. They will also show that the bounds established in Theorem A are attained for every $h\geq 1$ and $n\geq 0$.

First, we need a proposition.
\begin{proposition}
Let $G$ be an infinite countable locally finite group and consider a semi-direct product $\Gamma=G\rtimes_{\varphi} \mathbb{Z}$. Then  $\mathrm{cd}_{\mathbb{Q}}(\Gamma)=\underline{\mathrm{gd}}(\Gamma)=2$.
\end{proposition}
\begin{proof}
 Corollary \ref{cor: gd ext bounds}  implies that $\underline{\mathrm{gd}}(\Gamma)\leq 2$. On the other hand, we have $\mathrm{cd}_{\mathbb{Q}}(\Gamma)\leq \underline{\mathrm{cd}}(\Gamma)\leq \underline{\mathrm{gd}}(\Gamma)$. Therefore, it suffices to show that $\mathrm{cd}_{\mathbb{Q}}(\Gamma)\geq 2$.

Suppose, by a way of contradiction, that $\mathrm{cd}_{\mathbb{Q}}(\Gamma)\leq 1$. By Theorem 1.1 of \cite{Dunwoody}, $\Gamma$ acts on a tree $T$ without edge inversions and with finite vertex stabilizers. Moreover, Corollary 1.2 of \cite{Dunwoody} implies that $\Gamma$ is locally virtually cyclic. Now, fix a generator $t$ of $\mathbb{Z}$ such that $\Gamma=G\rtimes_{\varphi} \langle t \rangle$. Using the fact that $\Gamma$ is locally virtually cyclic, one can check that for each $g \in G$ there exists an $l \in \mathbb{N}_1$ such that $\varphi^l(g)=g$. Hence, for each $g \in G$ there exists an $l \in \mathbb{N}_1$ such that $t^l$ commutes with $g$ in $\Gamma$. Since $\langle t \rangle$ acts freely on the tree $T$, it follows from Proposition 24 of \cite{SerreTrees}, that $T$ contains an infinite straight line $L$ on which $t$ acts by translation with fixed amplitude. Moreover, if for some $k \in \mathbb{N}_1$, a subtree of $T$ is invariant under the action of $\langle t^k \rangle$, then it must contain $L$. Now, let $g \in G$ and choose $l \in \mathbb{N}_1$ such that $t^l$ commutes with $g$. Since $t^l$ and $g$ commute, the line $gL$ is invariant under the action of $\langle t^l \rangle$. It follows that $L$ is contained in $gL$, so we must have $gL=L$. This shows that $\Gamma$ acts on $L$. Since all vertex stabilizers are finite, $\Gamma$ has a finite normal subgroup $N$ such that $\Gamma/N$ is isomorphic to a discrete subgroup of the isometries of the real line. Because every discrete  subgroup of the isometry group of the real line is isomorphic to a subgroup of the infinite dihedral group, we conclude that $\Gamma$ is virtually cyclic. This contradicts the fact that $G$ is infinite locally finite. Therefore, we must have $\mathrm{cd}_{\mathbb{Q}}(\Gamma)\geq 2$.
\end{proof}

In the next three examples we will construct elementary amenable groups $G \rtimes_{\varphi_k} \mathbb{Z}$, for $k \in \{-1,0,1\}$ such that  $G$ is an infinite countable locally finite group and  \[\underline{\underline{\mathrm{gd}}}(G\rtimes_{\varphi_k} \mathbb{Z})= \underline{\mathrm{gd}}(G\rtimes_{\varphi_k} \mathbb{Z})+k.\]

\indent In our first example, every element $g \in G$ has a bounded orbit under the automorphism $\varphi_{-1} \in \mathrm{Aut}(G)$.
\begin{example} \rm Let $G$ be an infinite countable locally finite group and let $\Gamma$ be a semi-direct product $G \rtimes_{\varphi_{-1}} \mathbb{Z}$ that is locally bounded, meaning that for each $g \in G$, there exists some $l \in \mathbb{N}_1$ such that $\varphi_{-1}^l(g)=g$ ( e.g. $\varphi_{-1}=\mathrm{Id}$).
Since $\Gamma$ is not virtually cyclic, we have  $1  \leq  \underline{\underline{\mathrm{gd}}}(\Gamma)$. On the other hand, by  \cite[4.3]{DP}, we have $\underline{\underline{\mathrm{gd}}}(\Gamma)  \leq 1$. We conclude that $\underline{\underline{\mathrm{gd}}}(\Gamma)=1$.
\end{example}
In the second example, due to M.~Fluch, $\varphi_0 \in \mathrm{Aut}(G)$ acts freely on the conjugacy classes of elements of $G$.
\begin{example}\label{ex: wreath}  \rm Let $F$ be a non-trivial finite group and consider the infinite direct sum $G=\bigoplus_{i \in \mathbb{Z}} F$. One can form a semi-direct product $G \rtimes_{\varphi_0} \langle t \rangle$ by letting $t$ act on $G$ via the shifting automorphism (i.e. $t^l\cdot (a_i)_{i}=(a_{i-l})_{i}$ for all $l \in \mathbb{Z}$). Then $G \rtimes_{\varphi_0} \langle t \rangle$ is called the restricted wreath product of $F$ and $\mathbb{Z}$, which is denoted by $F \wr \mathbb{Z}$. In Example $18$ in \cite{Fluch}, Fluch shows that $\underline{\underline{\mathrm{gd}}}(F\wr \mathbb{Z})=2$.
\end{example}
In our third example, the automorphism $\varphi_1$ under consideration is a combination of the automorphisms of the two previous examples. This example also shows that the upper bound obtained in Theorem A of \cite{DP} is attained for $h=1$ and any $n\geq 0$.

\begin{example} \rm Let $n$ be a non-negative integer, $F$ be a non-trivial finite group, $K$ be a non-trivial finite abelian group. Take the infinite direct sums $A=\bigoplus_{i \in \mathbb{Z}} F$ and $B=\bigoplus_{\aleph_n} K$.
As before, consider the restricted wreath product $F\wr \mathbb{Z}=A\rtimes \mathbb{Z}$. Let $\Gamma=(F\wr \mathbb{Z})\oplus B$. Clearly, $\Gamma$ is an elementary amenable group of Hirsch length $h=1$ and cardinality $\aleph_n$. Next, we will prove that $\underline{\underline{\mathrm{gd}}}(\Gamma) =n+3$.

\begin{proof} First, we note that $\mbox{cd}_{\mathbb Q}(B)=n+1$ (see  \cite[6.10]{DicksKrophollerLearyThomas}). Let $M$ be a $\mathbb Q[B]$-module such that $\mathrm{H}^{n+1}(B,M)\neq 0$.  We can turn $M$ into a $\Gamma$-module, via projection of $\Gamma$ onto $B$. Let $C$ be an infinite cyclic subgroup of $F \wr \mathbb{Z}$. We claim that the map \[\mathrm{H}^{n+2}(\Gamma,M) \rightarrow  \mathrm{H}^{n+2}(\mathrm{Comm}_{\Gamma}[C] ,M),\] induced by the inclusion $\mathrm{Comm}_{\Gamma}[C] \subseteq \Gamma$, is injective.

First of all, note that we may assume that $C$ is maximal infinite cyclic. One can easily check that $\mathrm{Comm}_{\Gamma}[C]=\mathrm{Comm}_{F \wr \mathbb{Z}}[C]\oplus B$. It follows from Proposition 3.4.(a) of \cite{DP}, that $\mathrm{Comm}_{F \wr \mathbb{Z}}[C]$ is infinite cyclic. Since $\mathrm{Comm}_{F \wr \mathbb{Z}}[C]$ contains $C$ and $C$ is maximal infinite cyclic, we have $\mathrm{Comm}_{F \wr \mathbb{Z}}[C]=C$. We conclude that  $\mathrm{Comm}_{\Gamma}[C]=C\oplus B$. It now follows that $\Gamma$ has a finite index subgroup isomorphic to $(A \rtimes C) \oplus B$. A standard transfer-restriction argument now implies that the map \[\mathrm{H}^{n+2}(\Gamma,M) \rightarrow  \mathrm{H}^{n+2}((A \rtimes C) \oplus B,M),\] induced by the inclusion $(A \rtimes C) \oplus B \subseteq \Gamma$, is injective. Therefore, to prove the claim it suffices to show that the map \[\mathrm{H}^{n+2}((A \rtimes C) \oplus B,M) \rightarrow  \mathrm{H}^{n+2}(\mathrm{Comm}_{\Gamma}[C] ,M),\] induced by the inclusion $\mathrm{Comm}_{\Gamma}[C]=C\oplus B \subseteq (A \rtimes C) \oplus B $ is injective.

Let $T$ be a $\mathbb Q[A\rtimes C]$-module  with trivial $A$-action. Since $T$ is torsion-free and $A$ is a torsion group that acts trivially on $T$, we have $\mathrm{H}^1(A, T)=0$. Using this together with the fact that $\mathrm{cd}_{\mathbb{Q}}(C)$ and $\mathrm{cd}_{\mathbb{Q}}(A)$ are both $1$, an application of the Lyndon-Hochschild-Serre spectral sequence  associated to the extension \[1 \rightarrow A \rightarrow A \rtimes C \rightarrow C \rightarrow 1,\] shows that $\mathrm{H}^2(A \rtimes C, T)=0$ for every $\mathbb Q[A\rtimes C]$-module $T$ with trivial $A$-action. These observations, together with the fact that $\mathrm{cd}_{\mathbb{Q}}(B)= n+1$ and the naturality between the Lyndon-Hochschild-Serre spectral sequences with coefficients in $M$ associated to the extensions \[1 \rightarrow B \rightarrow (A \rtimes C)\oplus B \rightarrow A \rtimes C \rightarrow 1\] and \[1 \rightarrow B \rightarrow \mathrm{Comm}_{\Gamma}[C] \rightarrow C \rightarrow 1,\] yield a commutative diagram
\[ \xymatrix{  \mathrm{H}^{n+2}((A \rtimes C) \oplus B,M) \ar[r] \ar[d]^{\cong} & \mathrm{H}^{n+2}(\mathrm{Comm}_{\Gamma}[C] ,M) \ar[d]^{\cong} \\
                \mathrm{H}^1(A \rtimes C,\mathrm{H}^{n+1}(B,M)) \ar[r] &   \mathrm{H}^{1}(C,\mathrm{H}^{n+1}(B,M)),} \]
where the horizontal maps are induced by the inclusions $C \subset A \rtimes C$ and  $\mathrm{Comm}_{\Gamma}[C] \subseteq (A \rtimes C)\oplus B$ and the vertical maps are isomorphisms arising form the spectral sequences. The Lyndon-Hochschild-Serre spectral sequence with coefficients in $\mathrm{H}^{n+1}(B,M)$ associated to the extension \[1 \rightarrow A \rightarrow A \rtimes C \rightarrow C \rightarrow 1\] implies that we have a short exact sequence
\[ 0 \rightarrow \mathrm{H}^1(C,\mathrm{H}^{n+1}(B,M)) \rightarrow  \mathrm{H}^1(A \rtimes C,\mathrm{H}^{n+1}(B,M)) \rightarrow  \mathrm{H}^0(C,\mathrm{H}^1(A,\mathrm{H}^{n+1}(B,M))) \rightarrow 0.\]
Since $\mathrm{H}^{n+1}(B, M)$ is torsion-free and $A$ is a torsion group that acts trivially on $\mathrm{H}^{n+1}(B,M)$, we have $\mathrm{H}^1(A,\mathrm{H}^{n+1}(B,M))=0$ and therefore \[ \mathrm{H}^1(A \rtimes C,\mathrm{H}^{n+1}(B,M))\cong \mathrm{H}^1(C,\mathrm{H}^{n+1}(B,M)).\] Using this, we obtain a commutative diagram
\[ \xymatrix{  \mathrm{H}^{n+2}((A \rtimes C) \oplus B,M) \ar[r] \ar[d]^{\cong} & \mathrm{H}^{n+2}(\mathrm{Comm}_{\Gamma}[C] ,M) \ar[d]^{\cong} \\
                \mathrm{H}^1(C,\mathrm{H}^{n+1}(B,M)) \ar[r]^{\mathrm{Id}} &   \mathrm{H}^1(C,\mathrm{H}^{n+1}(B,M)).} \]
that proves our claim.

Since $\mathrm{H}^{\ast}(F,M)=0$ for all finite subgroups $F$ of $\Gamma$, we may conclude from Theorem $6.1$ in \cite{Martinez} that there are natural isomorphisms $$\mathrm{H}^{\ast}_{\mathcal{F}}(\Gamma,M^{-})\cong \mathrm{H}^{\ast}(\Gamma,M)$$ and $$\mathrm{H}^{\ast}_{\mathcal{F}\cap \mathrm{Comm}_{\Gamma}[C]}(\mathrm{Comm}_{\Gamma}[C],M^{-})\cong \mathrm{H}^{\ast}(\mathrm{Comm}_{\Gamma}[C],M).$$ Hence, for every infinite cyclic subgroup $C$ of $F \wr \mathbb{Z}$, the map
\[ i^{n+2}(C): \mathrm{H}^{n+2}_{\mathcal{F}}(\Gamma,M^{-}) \rightarrow \mathrm{H}^{n+2}_{\mathcal{F}\cap \mathrm{Comm}_{\Gamma}[C]}(\mathrm{Comm}_{\Gamma}[C],M^{-}) \]
is injective and \[\mathrm{H}^{n+2}_{\mathcal{F}\cap \mathrm{Comm}_{\Gamma}[C]}(\mathrm{Comm}_{\Gamma}[C],M^{-})\cong  \mathrm{H}^1(C,\mathrm{H}^{n+1}(B,M))\cong \mathrm{H}^{n+1}(B,M)\neq 0.\]
\indent Next, observe that $\mbox{cd}_{\mathbb Q}(\Z\oplus B)=n+2$ and since $\Z\oplus B$  is a subgroup of $\Gamma$, it follows that $\mbox{cd}_{\mathbb Q}(\Gamma)\geq n+2$. But, by Corollary \ref{cor: useful combination}, we have that $\underline{\mbox{gd}}(\Gamma)\leq n+2$. Therefore, $\underline{\mbox{gd}}(\Gamma)=n+2$.

We also note that from Corollary \ref{cor: useful combination}(ii) and  the proof of Theorem A (Case 1: $h=1$) that $\mathrm{gd}_{\mathcal{F}[H]}(\mathrm{Comm}_{\Gamma}[H])\leq n+1$ for all infinite virtually cyclic subgroups $H$ of $\Gamma$.\\
\indent Let $\mathcal{H}$ be the family of virtually cyclic subgroups of $\Gamma$. Proposition \ref{prop: mayer-vietoris} gives us an exact sequence
\[ \mathrm{H}^{n+2}_{\mathcal{F}}(\Gamma, M^{-}) \xrightarrow{\prod i^{n+2}(H)} \prod_{[H] \in \mathcal{I}}  \mathrm{H}^{n+2}_{\mathcal{F}\cap \mathrm{Comm}_{\Gamma}[H]}(\mathrm{Comm}_{\Gamma}[H],M^{-}) \rightarrow  \mathrm{H}^{n+3}_{\mathcal{H}}(\Gamma,M^{-}) \rightarrow 0. \]
Since the set $\mathcal{I}$ contains more than one element represented by an infinite cyclic subgroup of $F \wr \mathbb{Z}$, the map $\prod i^{n+2}(H)$ cannot be surjective because $i^{n+2}(H)$ is injective when $H$ is a subgroup $F \wr \mathbb{Z}$. Since $$ \prod_{[H] \in \mathcal{I}}  \mathrm{H}^{n+2}_{\mathcal{F}\cap \mathrm{Comm}_{\Gamma}[H]}(\mathrm{Comm}_{\Gamma}[H],M^{-})\neq 0,$$ it follows that $\mathrm{H}^{n+3}_{\mathcal{H}}(\Gamma,M^{-})\neq 0$. Therefore, we have $\underline{\underline{\mathrm{gd}}}(\Gamma)\geq n+3$. Theorem A now implies that $\underline{\underline{\mathrm{gd}}}(\Gamma)= n+3$.\end{proof}
\end{example}

It is well-known and not difficult to see that, for any two groups $\Gamma_1$ and $\Gamma_2$, one has  \[\underline{\mathrm{gd}}(\Gamma_1 \times \Gamma_2) \leq  \underline{\mathrm{gd}}(\Gamma_1) + \underline{\mathrm{gd}}(\Gamma_2).\] When considering the family of virtually cyclic subgroups, the bounds are not as clear when taking products. As pointed out in \cite{LuckWeiermann}, Corollary $5.6$ and Remark $5.7$, the best one can do in general is
\[ \underline{\underline{\mathrm{gd}}}(\Gamma_1 \times \Gamma_2) \leq  \underline{\underline{\mathrm{gd}}}(\Gamma_1) + \underline{\underline{\mathrm{gd}}}(\Gamma_2)+3.\]
Next, we will consider two examples that will again illustrate the somewhat strange behaviour of this geometric dimension when taking products of groups. 
\begin{example} \label{ex: direct sum}\rm   {Let $F$ be a non-trivial finite abelian group. Then for any integers $n\geq 0$ and $h\geq 2$, we have $\underline{\underline{\mathrm{gd}}}\Big((\bigoplus_{\aleph_n} F) \oplus \mathbb{Z}^h\Big)=n+h+2$ and $\underline{\underline{\mathrm{gd}}}\Big((\bigoplus_{\aleph_n} F) \oplus \mathbb{Z}\Big)=n+1$. }
\begin{proof}  Let $\Gamma=(\bigoplus_{\aleph_n} F)\oplus \mathbb{Z}^h$. Then $\Gamma$ is elementary amenable of Hirsch length $h$ and cardinality $\aleph_n$.   We will prove that $\underline{\underline{\mathrm{gd}}}(\Gamma) = n+h+2$.
By Theorem A, it suffices to show that $\underline{\underline{\mathrm{cd}}}(\Gamma) \geq n+h+2$.

By Corollary 6.10 of \cite{DicksKrophollerLearyThomas}, $\mbox{cd}_{\mathbb Q}(\bigoplus_{\aleph_n} F)=n+1$. So, a corner argument of the Lyndon-Hochschild-Serre spectral sequence associated to the extension with $\bigoplus_{\aleph_n} F$ as its kernel, implies  $\mathrm{cd}_{\mathbb{Q}}(\Gamma)=n+h+1$. Since $\underline{\mathrm{gd}}(\Gamma)\leq n+h+1$ (see Corollary \ref{cor: useful combination}(ii) and \cite[Corollary 4]{FloresNuc}), this implies that $\underline{\mathrm{gd}}(\Gamma)=n+h+1$.
Note that for every infinite cyclic subgroup $H$ of $\Gamma$,  $\mathrm{Comm}_{\Gamma}[H]=\Gamma$. Also, by Lemma \ref{lemma: virt poly}, we have that $\mathrm{gd}_{\mathcal{F}[H]}(\mathrm{Comm}_{\Gamma}[H])\leq n+h.$

Next, let $\mathcal{H}$ be the family of virtually cyclic subgroups of $\Gamma$ and let $\mathcal{F}$ be the family of finite subgroups of $\Gamma$. Since $\underline{\mathrm{gd}}(\Gamma)=n+h+1$, we can find a module $M \in  \mbox{Mod-}\mathcal{O}_{\mathcal{F}}\Gamma$ such that $ \mathrm{H}^{n+h+1}_{\mathcal{F}}(\Gamma,M)\neq 0$. Denote the induction functor from $\mbox{Mod-}\mathcal{O}_{\mathcal{F}}\Gamma$ to $\mbox{Mod-}\mathcal{O}_{\mathcal{H}}\Gamma$ by $\mathrm{ind}_{\mathcal{F}}^{\mathcal{H}}$ and note that by Lemma \ref{lemma: induction} $\mathrm{ind}_{\mathcal{F}}^{\mathcal{H}}(M)\cong M$ as $\mathcal{O}_{\mathcal{F}}\Gamma$-modules. Proposition \ref{prop: mayer-vietoris} gives an exact sequence
\begin{equation*} \label{eq: mv 3} \Big(\prod_{[H] \in \mathcal{I}} \mathrm{H}^{n+h+1}_{\mathcal{F}[H] }( \mathrm{Comm}_{\Gamma}[H],\mathrm{ind}_{\mathcal{F}}^{\mathcal{H}}(M))\Big)\oplus  \mathrm{H}^{n+h+1}_{\mathcal{F}}(\Gamma,M) \rightarrow  \prod_{[H] \in \mathcal{I}} \mathrm{H}^{n+h+1}_{\mathcal{F}\cap  \mathrm{Comm}_{\Gamma}[H] }( \mathrm{Comm}_{\Gamma}[H],M)  \end{equation*} \[
 \rightarrow \mathrm{H}^{n+h+2}_{\mathcal{H}}(\Gamma,\mathrm{ind}_{\mathcal{F}}^{\mathcal{H}}(M)) \rightarrow \Big(\prod_{[H] \in \mathcal{I}} \mathrm{H}^{n+h+2}_{\mathcal{F}[H] }( \mathrm{Comm}_{\Gamma}[H],\mathrm{ind}_{\mathcal{F}}^{\mathcal{H}}(M))\Big)\oplus  \mathrm{H}^{n+h+2}_{\mathcal{F}}(\Gamma,M)  \]
which, by the above established facts, transforms to
\begin{equation*}   \mathrm{H}^{n+h+1}_{\mathcal{F}}(\Gamma,M) \xrightarrow{\Delta} \prod_{[H] \in \mathcal{I}}\mathrm{H}^{n+h+1}_{\mathcal{F}}( \Gamma,M)
 \rightarrow \mathrm{H}^{n+h+2}_{\mathcal{H}}(\Gamma,\mathrm{ind}_{\mathcal{F}}^{\mathcal{H}}(M)) \rightarrow 0.
\end{equation*}
Since $\mathrm{H}^{n+h+1}_{\mathcal{F}}( \Gamma,M)$ is non-zero, $\Delta$ is the diagonal map and $\mathcal{I}$ contains more than one element (because $h\geq 2$), it follows that $\Delta$ cannot be surjective. This implies that $\mathrm{H}^{n+h+2}_{\mathcal{H}}(\Gamma,\mathrm{ind}_{\mathcal{F}}^{\mathcal{H}}(M))\neq 0$ and hence $\underline{\underline{\mathrm{cd}}}(\Gamma) \geq n+h+2$. We conclude that $\underline{\underline{\mathrm{gd}}}(\Gamma) = n+h+2$.

Let $G=(\bigoplus_{\aleph_n} F) \oplus \mathbb{Z}$. Note  that all infinite cyclic subgroups  of $G$ are equivalent and $G=\mathrm{Comm}_{G}[\Z]$. Using Lemma \ref{lemma: virt poly}, it follows that \[\underline{\underline{\mathrm{gd}}}(G)=\mathrm{gd}_{\mathcal{F}[\Z]}(\mathrm{Comm}_{G}[\Z])= \underline{\mbox{gd}}(\bigoplus_{\aleph_n} F).\] Since $\mbox{cd}_{\mathbb Q}(\bigoplus_{\aleph_n} F)=n+1$, applying Theorem \ref{th: fin gen subgroups} yields  $\underline{\mbox{gd}}(\bigoplus_{\aleph_n} F)=n+1$.  Hence, $\underline{\underline{\mathrm{gd}}}(G)=n+1$.
\end{proof}
\end{example}

\begin{example} \label{ex: wreath plus Z}  \rm  {Let $F$ be a non-trivial finite  group. Then $\underline{\underline{\mathrm{gd}}}\Big((F \wr \mathbb{Z}) \oplus \mathbb{Z}\Big)=3$ while $\underline{\underline{\mathrm{gd}}}(F \wr \mathbb{Z})=2$ and $\underline{\underline{\mathrm{gd}}}(\mathbb{Z})=0$.}
\begin{proof} Let $\Gamma=(F \wr \mathbb{Z}) \oplus \mathbb{Z}$. Since $\mathrm{cd}_{\mathbb{Q}}(F \wr \mathbb{Z})=2$, one can deduce that $\underline{\mathrm{gd}}(\Gamma)=3$ by showing that $\mathrm{cd}_{\mathbb{Q}}(\Gamma)=3$ using a spectral sequence argument.  We claim that $\underline{\underline{\mathrm{gd}}}(\Gamma)=3$.
Note that $\mathbb{Z}^2$ is a subgroup of $\Gamma$, therefore $\underline{\underline{\mathrm{gd}}}(\Gamma) \geq 3$. Hence, to prove the claim it suffices to show that $\underline{\underline{\mathrm{cd}}}(\Gamma) \leq 3$. Let $\mathcal{H}$ be the family of virtually cyclic subgroups of $\Gamma$ and let $\mathcal{F}$ be the family of finite subgroups of $\Gamma$. Let $M \in \mbox{Mod-}\mathcal{O}_{\mathcal{H}}\Gamma$. By Proposition \ref{prop: mayer-vietoris}, we have an exact sequence
\begin{equation} \label{eq: mv 1} \Big(\prod_{[H] \in \mathcal{I}} \mathrm{H}^{3}_{\mathcal{F}[H] }( \mathrm{Comm}_{\Gamma}[H],M)\Big)\oplus  \mathrm{H}^{3}_{\mathcal{F}}(\Gamma,M) \rightarrow  \prod_{[H] \in \mathcal{I}} \mathrm{H}^{3}_{\mathcal{F}\cap  \mathrm{Comm}_{\Gamma}[H] }( \mathrm{Comm}_{\Gamma}[H],M)  \end{equation} \[
 \rightarrow \mathrm{H}^{4}_{\mathcal{H}}(\Gamma,M) \rightarrow \Big(\prod_{[H] \in \mathcal{I}} \mathrm{H}^{4}_{\mathcal{F}[H] }( \mathrm{Comm}_{\Gamma}[H],M)\Big)\oplus  \mathrm{H}^{4}_{\mathcal{F}}(\Gamma,M)  .\]
\indent Let $H$ be an infinite cyclic subgroup of $\Gamma$. If $H$ is contained in the direct summand $\mathbb{Z}$ of $\Gamma$, then clearly $\mathrm{Comm}_{\Gamma}[H]=\Gamma$. Using Corollary 5.2 in \cite{Martinez} one checks that $\mathrm{cd}_{\mathcal{F}[H]}(\mathrm{Comm}_{\Gamma}[H])\leq \underline{\mathrm{cd}}(F \wr \mathbb{Z})$. Hence, we have $\mathrm{cd}_{\mathcal{F}[H]}(\mathrm{Comm}_{\Gamma}[H])\leq \underline{\mathrm{gd}}(F \wr \mathbb{Z})=2$. Furthermore, there is exactly one $[H] \in \mathcal{I}$ such that $H$ is an infinite cyclic subgroup of $\Gamma$ that is contained in the direct summand $\mathbb{Z}$ of $\Gamma$. \\
\indent Now suppose $H$ is an infinite cyclic subgroup of $\Gamma$ that is not equivalent to the direct summand $\mathbb{Z}$ of $\Gamma$. One can verify, via direct computation or by using Proposition 3.4(a) of \cite{DP}, that $\mathrm{Comm}_{\Gamma}[H] \cap (F \wr \mathbb{Z}) \cong \mathbb{Z}$.  This implies that $\mathrm{Comm}_{\Gamma}[H]\cong \mathbb{Z}^2$, and hence $\underline{\mathrm{cd}}(\mathrm{Comm}_{\Gamma}[H])=2$. Let $C$ be the maximal infinite cyclic subgroup of $\mathrm{Comm}_{\Gamma}[H]$ containing $H$. The subgroups in $\mathcal{F}[H]$ are exactly the subgroups of $C$, and $\mathrm{Comm}_{\Gamma}[H]/C \cong \mathbb{Z}$. Hence, we have \[\mathrm{cd}_{\mathcal{F}[H]}(\mathrm{Comm}_{\Gamma}[H])\leq \mathrm{cd}(\mathrm{Comm}_{\Gamma}[H]/C)= \mathrm{cd}(\mathbb{Z})=1\] by applying Corollary 5.2 in \cite{Martinez} to the extension $0 \rightarrow C \rightarrow \mathrm{Comm}_{\Gamma}[H] \rightarrow \mathbb{Z} \rightarrow 0$, with the appropriate families. Since $\underline{\mathrm{gd}}(\Gamma)=3$, the exact sequence $(\ref{eq: mv 1})$ becomes
\begin{equation*}   \mathrm{H}^{3}_{\mathcal{F}}(\Gamma,M) \xrightarrow{\mathrm{Id}} \mathrm{H}^{3}_{\mathcal{F}}( \Gamma,M)
 \rightarrow \mathrm{H}^{4}_{\mathcal{H}}(\Gamma,M) \rightarrow 0.
\end{equation*}
This implies that $\mathrm{H}^{4}_{\mathcal{H}}(\Gamma,M)=0$, so $\underline{\underline{\mathrm{cd}}}(\Gamma) \leq 3$. We conclude that $\underline{\underline{\mathrm{gd}}}(\Gamma) = 3$.\end{proof}
\end{example}

\section{Extensions with torsion-free quotient}
The following corollary follows easily from Theorem A and Corollary \ref{cor: gd ext bounds}.
\begin{corollary}  Consider a short exact sequence of groups $N \rightarrow \Gamma \to Q$
such that $N$ is elementary amenable of finite Hirsch length and of cardinality $\aleph_n$ and such that $\underline{\underline{\mathrm{gd}}}(Q) < \infty$. Then
\[  \underline{\underline{\mathrm{gd}}}(\Gamma) \leq \underline{\underline{\mathrm{gd}}}(Q)+n+h(N)+3. \]
\end{corollary}
In general, given a short exact sequence of groups $N \rightarrow \Gamma \rightarrow Q$, one can ask whether or not the finiteness of $\underline{\underline{\mathrm{gd}}}(N)$ and $\underline{\underline{\mathrm{gd}}}(Q)$ implies the finiteness of $\underline{\underline{\mathrm{gd}}}(\Gamma)$.
The answer to this question is unknown in general. Theorem B gives a positive answer  when $Q$ is torsion-free and the number $\delta(N)$ is finite.

We start with two lemmas.
\begin{lemma}\label{lemma: delta} Let $N$ be a group of cardinality $\aleph_n$, then $\underline{\mathrm{gd}}(N)\leq \delta(N)+1$.
\end{lemma}
\begin{proof}
This is immediate from the definition of $\delta(N)$ and Corollary \ref{cor: useful combination}(i).
\end{proof}
\begin{lemma} \label{lemma: fam[H]} Let $H$ be an infinite virtually cyclic subgroup of a group $\Gamma$, then
\[ \mathrm{cd}_{\mathcal{F}[H]}(\mathrm{Comm}_{\Gamma}[H]) \leq \max\{\underline{\underline{\mathrm{cd}}}(\mathrm{Comm}_{\Gamma}[H]),\underline{\mathrm{cd}}(\mathrm{Comm}_{\Gamma}[H])\}.\]
\end{lemma}
\begin{proof} Let $G=\mathrm{Comm}_{\Gamma}[H]$ and denote by $\mathcal{H}$ the family of virtually cyclic subgroups of $G$. Take $K \in \mbox{Mod-}\mathcal{O}_{\mathcal{F}[H]}G$ and denote the induction functor from $\mbox{Mod-}\mathcal{O}_{\mathcal{F}[H]}G$ to $\mbox{Mod-}\mathcal{O}_{\mathcal{H}}G$ by $\mathrm{ind}_{\mathcal{F}[H]}^{\mathcal{H}}$. Let $M=\mathrm{ind}_{\mathcal{F}[H]}^{\mathcal{H}}(K)$. By Proposition \ref{prop: mayer-vietoris}, we have an exact sequence
\[ \mathrm{H}^{n}_{\mathcal{H}}(G,M) \rightarrow \Big(\prod_{[S] \in \mathcal{I}} \mathrm{H}^{n}_{\mathcal{F}[S] }( \mathrm{Comm}_{G}[S],M)\Big)\oplus  \mathrm{H}^{n}_{\mathcal{F}}(G,M) \rightarrow  \prod_{[S] \in \mathcal{I}} \mathrm{H}^{n}_{\mathcal{F}\cap  \mathrm{Comm}_{G}[S] }( \mathrm{Comm}_{G}[S],M).  \]
Note that $[H]\in \mathcal{I}$, and that $\mathrm{Comm}_{G}[H]=G$. Hence, if we choose $n>\max\{\underline{\underline{\mathrm{cd}}}(G),\underline{\mathrm{cd}}(G)\}$, it follows from the exact sequence and Lemma \ref{lemma: induction} that $\mathrm{H}^{n}_{\mathcal{F}[H] }( G,M)=\mathrm{H}^{n}_{\mathcal{F}[H] }( G,K)=0$. Since this holds for any $K \in \mbox{Mod-}\mathcal{O}_{\mathcal{F}[H]}G$, we conclude that   \[\mathrm{cd}_{\mathcal{F}[H]}(\mathrm{Comm}_{G}[H]) \leq \max\{\underline{\underline{\mathrm{cd}}}(\mathrm{Comm}_{\Gamma}[H]),\underline{\mathrm{cd}}(\mathrm{Comm}_{\Gamma}[H])\}.\]
\end{proof}
We can now prove Theorem B.
\begin{proof}[Proof of Theorem B] It suffices to prove that for any semi-direct product $N\rtimes \mathbb{Z}$, we have $$\underline{\underline{\mathrm{gd}}}(N\rtimes \mathbb{Z})\leq \max\{\underline{\underline{\mathrm{gd}}}(N), \delta(N)\}+2.$$ Because then it would follow from Corollary \ref{cor: gd ext bounds} that \[\underline{\underline{\mathrm{gd}}}(\Gamma) \leq \underline{\underline{\mathrm{gd}}}(Q)+ \max\{\underline{\underline{\mathrm{gd}}}(N), \delta(N)\}+2,\] since $Q$ is torsion-free.

If $N$ is virtually cyclic, then the desired inequality clearly holds. Assume that $N$ is not virtually cyclic. Then $\underline{\underline{\mathrm{gd}}}(N)\geq 1$, so it suffices to prove that \[\underline{\underline{\mathrm{cd}}}(N\rtimes \mathbb{Z})\leq  \max\{\underline{\underline{\mathrm{gd}}}(N), \delta(N)\}+2.\]
Let us consider a semi-direct product $G=N\rtimes \mathbb{Z}$ and an infinite cyclic subgroup $H$ of $G$. There is a short exact sequence
\[ 1 \rightarrow N_0 \rightarrow \mathrm{Comm}_{G}[H] \xrightarrow{p} m\mathbb{Z} \rightarrow 0 \]
where $N_0=N\cap \mathrm{Comm}_{G}[H]$  and $m$ is an integer that can be zero. We will find an upper bound for $\mathrm{cd}_{\mathcal{F}[H]}(\mathrm{Comm}_{G}[H])$. To do this, we distinguish between two cases.

First, assume that $H$ is contained in $N$. If $K$ is a virtually cyclic subgroup of $\mathrm{Comm}_{G}[H]$ which is equivalent to $H$ then $p(K)$ is a finite group, since $H$ is contained in the kernel of $p$. Since $m\mathbb{Z}$ is torsion-free, we conclude that every subgroup of $\mathcal{F}[H]$ is contained in $N_0$. Since $N_0=\mathrm{Comm}_{N_0}[H]$, it follows from Lemma \ref{lemma: fam[H]} that \[\mathrm{cd}_{\mathcal{F}[H]}(N_0)\leq \max\{\underline{\underline{\mathrm{cd}}}(N_0),\underline{\mathrm{cd}}(N_0)\}\leq\max\{\underline{\underline{\mathrm{gd}}}(N),\underline{\mathrm{gd}}(N)\}.\] Hence, Corollary \ref{cor: gd ext bounds} implies that \[\mathrm{cd}_{\mathcal{F}[H]}(\mathrm{Comm}_{G}[H])\leq \max\{\underline{\underline{\mathrm{gd}}}(N),\underline{\mathrm{gd}}(N)\}+1.\]

Now, assume that $H$ is not contained in $N$ (in this case $m$ is not zero). Let $K$ be a finitely generated subgroup of $\mathrm{Comm}_{G}[H]$ that contains $H$. By Proposition \ref{prop: finitely generated}, $K$ is a $B$-by-$\mathbb{Z}$, where $B$ is a finitely generated subgroup of $N$, and $H$ has an infinite cyclic subgroup $R$ that is central in $K$. It follows from Lemma \ref{lemma: virt poly} that $\mathrm{gd}_{\mathcal{F}[H]\cap K}(K)= \underline{\mathrm{gd}}(K/R)$. Since $K/R$ is $B$-by-finite cyclic where $B$ is a finitely generated subgroup of $N$, it follows that $\mathrm{gd}_{\mathcal{F}[H]\cap K}(K)\leq \delta(N)-n$. Hence, by Corollary \ref{cor: useful combination} we have $\mathrm{gd}_{\mathcal{F}[H]}(\mathrm{Comm}_{G}[H])\leq \delta(N)+1$. Therefore, in general we have \[\mathrm{cd}_{\mathcal{F}[H]}(\mathrm{Comm}_{G}[H])\leq \max\{\underline{\underline{\mathrm{gd}}}(N),\underline{\mathrm{gd}}(N), \delta(N)\}+1.\]
\indent Using Theorem \ref{th: cd push out} together with Lemma \ref{lemma: delta}, we deduce that \[\underline{\underline{\mathrm{cd}}}(N\rtimes \mathbb{Z})\leq \max\{\underline{\underline{\mathrm{gd}}}(N), \delta(N)\}+2.\]
\end{proof}
As mentioned in the introduction, we do not know of an example of a group $N$ such that $\underline{\mathrm{gd}}(N)< \infty$ and $\delta(N)=\infty$. However, we will show that there are groups $N$  for which $\delta(N)$ is finite but the difference $\delta(N)-\underline{\mathrm{gd}}(N)$ is  arbitrarily large. This groups will be constructed  by the means of Theorem C.

Before going into the proof of this theorem, let us first recall some of the basics of the general construction of Bestvina-Brady groups $H_L$ and the semi-direct products $H_L\rtimes Q$. For more details, we suggest the original sources  \cite{BestvinaBrady} and \cite{LearyNucinkis}. 

For a given  finite flag complex $L$, the right-angled Artin group $G_L$ is the group whose generators are in one-to-one correspondence with the
vertex set of $L$ subject only to the relations that if two vertices are connected by an edge in the $1$-skeleton of $L$, they commute.  There is an epimorphism $\phi_L$: $G_L\to \Z$ defined by mapping each of the given generators to $1$. The group $H_L$ is defined to be the kernel of $\phi_L$.
It is a general fact that any right-angled Artin group can be embedded into a right angled Coxeter group which in turn is known to be a subgroup of $\mbox{GL}_m(\Z)$ for some $m$ (see \cite[Corollary 8]{LearyNucinkis}).  This shows that  $H_L$ is an integral linear group.  The map sending $L$ to $G_L$ is a functor from flag complexes to groups. Therefore, every automorphism of $L$ induces an automorphism of $G_L$. In particular, if a group $Q$ acts admissibly and simpicially on $L$, this induces a representation of $Q$ into the group of automorphisms of $G_L$.  Since the epimorphism $\phi_L$ is $Q$-equivariant, this representation restricts to $H_L$ and one can form the semi-direct product $H_L\rtimes Q$.

\begin{proof}[Proof of Theorem C]  

(a). Let $L$ be an acyclic $n$-dimensional finite flag complex with a simplicial admissible faithful action of a cyclic group $Q$ of a prime order. It follows from the Main Theorem of \cite{BestvinaBrady} that $H_L$ is of type $FP$ and if one imposes the stronger condition that $L$ is contractible, then $H_L$ is of type $F$. Hence $\Gamma=H_{L}\rtimes Q$ is of type $VFP$ in general, and of type $VF$ if $L$ is contractible. Since $L$ is acyclic, it follows from Theorem 22 of \cite{LearySaad} that $\mathrm{cd}(H_L)=n$. Since $H_L$ contains $\mathbb{Z}^n$ as a subgroup, we also have $\mathrm{hd}(H_L)=n$. Denote by $C_{\Gamma}(Q)$ the centralizer of $Q$ in $\Gamma$ and set $C_{H_L}(Q)=C_{\Gamma}(Q)\cap H_L$. Note that the normalizer $N_{\Gamma}(Q)$ of $Q$ in ${\Gamma}$ is isomorphic to $C_{H_L}(Q)\times Q$ and that the Weyl group $W(Q)=N_{\Gamma}(Q)/Q$ is isomorphic to $C_{H_L}(Q)$. Since, by Smith Theory (see \cite{Smith}), the fixed-point set  $L^{Q}$  is non-empty, from Theorem 3(2) of \cite{LearyNucinkis}, it follows that all non-trivial finite subgroups of $\Gamma$ are conjugate to $Q$ and that $C_{H_L}(Q)$ is the Bestvina-Brady group $H_{L^Q}$.

Recall that in our case $\Gamma_s=\prod_{i=1}^s H_{L_i}\rtimes C_{p_i}$ where $H_{L_i}$ is the Bestvina-Brady group associated to the $n_i$-dimensional acyclic finite flag complex $L_i$ for  $i=1, \dots, s$. So, in the above,  replacing  $L$ and $Q$ with $L_i$ and $C_{p_i}$, respectively, for each $i$, we obtain that $H_{L_i}$ is an $FP$-group, $\Z^{n_i}$ is a subgroup of $H_{L_i}$, and  $\mathrm{hd}(H_{L_i})=\mathrm{cd}(H_{L_i})=n_i$.
This implies that $\Gamma_s$ is of type $VFP$ ($VF$ if $L$ is contractible) and $\mathrm{vhd}(\Gamma_s)=\mathrm{vcd}(\Gamma)=\sum_{i=1}^s n_i$.  Note that each group $H_{L_i}\rtimes C_{p_i}$ is an integral linear group because it is a finite extension of such a group.  Since $\G_s$ is a finite product of integral linear groups, it is itself integral linear.

For the rest of the proof we argue by induction on $s$. Suppose $s=1$. Denote $p=p_1$, $n_1=n$, $L=L_1$ and $\Gamma=\Gamma_1$. Let $\mathcal{F}$ be the family of finite subgroups of $\Gamma$, which in this case consists of the trivial subgroup and subgroups isomorphic to  $C_{p}$. Let $\mathcal{T}$ be the family containing only the trivial subgroup of $\Gamma$. Define the covariant functor 
\[ V: \mathcal{O}_{\mathcal{F}}\Gamma \rightarrow \mathbb{Z}\mbox{-mod}: \Gamma /H \mapsto 
\left\{\begin{array}{lll}  \F_q & \mbox{if}&  H=\{e\} \\
                    0            & \mbox{if}&  H\neq \{e\}, \end{array} \right. \]
where every morphism $\Gamma /H \xrightarrow{g} \Gamma /K$ with $K \neq \{e\}$ is mapped to the zero homomorphism and every morphism $\Gamma /\{e\} \xrightarrow{g} \Gamma /\{e\}$ is mapped to the identity homomorphism. 
We define a strong equivalence relation on $\mathcal{F}\smallsetminus \mathcal{T}$ by $$H\sim K \Leftrightarrow H=K.$$ Under this equivalence relation, one has $N_{\Gamma}[H]=C_{\Gamma}(H)$, the family $\mathcal{T}[H]$ is $\{ \{e\},H  \}$, and there is exactly one conjugacy class of equivalence classes of non-trivial subgroups of $\G$, namely the one represented by $C_{p}$. Now, the $\mathrm{Tor}$ long exact sequence of Theorem \ref{th: ext} with $M=\underline{\mathbb{Z}}$ and $T=V$ yields the exact sequence
\[ \ldots \rightarrow \mathrm{H}_{n+1}^{\mathcal{F}}(\Gamma,V) \rightarrow \mathrm{H}_{n}(C_{\Gamma}(C_{p}),\F_q) \rightarrow \mathrm{H}_{n}^{\mathcal{T}[C_{p}]}(C_{\Gamma}(C_{p}),V)\oplus \mathrm{H}_{n}(\Gamma,\F_q) \rightarrow \ldots \ .                \]
Since  $C_{\Gamma}(C_{p})\cong C_{H_{L}}(C_{p})\times C_{p}$ and $p\neq q$, it follows that $$\mathrm{H}_{n}(C_{\Gamma}(C_{p}),\F_q)=\mathrm{H}_{n}(C_{H_{L}}(C_{p}),\F_q).$$ 
We claim that $\mathrm{H}_{n}^{\mathcal{T}[C_{p}]}(C_{\Gamma}(C_{p}),V)=0$. To see this, consider the convergent $E_2$-term spectral sequence
\[  \mathrm{H}_t(C_{H_{L}}(C_{p}),\mathrm{H}_r^{\mathcal{T}[C_{p}]}(C_p,V)) \Longrightarrow \mathrm{H}_{t+r}^{\mathcal{T}[C_{p}]}(C_{\Gamma}(C_{p}),V)   \]
associated to the extension $$1 \rightarrow C_{p} \rightarrow C_{\Gamma}(C_{p}) \rightarrow C_{H_{L}}(H) \rightarrow 1$$ (see  \cite[Theorem 5.1]{Martinez}). Since $C_{p} \in \mathcal{T}[C_{p}]$, we have $\mathrm{H}_r^{\mathcal{T}[C_{p}]}(C_{p},V)=0$ for all $r \geq 1$ and $\mathrm{H}_0^{\mathcal{T}[C_{p}]}(C_{p},V)=V(\Gamma/C_{p})=0$. Hence, all the $E_2$-terms vanish which proves the claim. We conclude that there is an exact sequence
\[ \mathrm{H}_{n+1}^{\mathcal{F}}(\G,V) \rightarrow \mathrm{H}_{n}(C_{H_{L}}(C_{p}),\F_q) \rightarrow \mathrm{H}_{n}(\Gamma,\F_q). \] 
Since $\Gamma$ is of type ${FP}_\infty$,  the group $\mathrm{H}_{n}(\Gamma, \F_q)$ is a finite dimensional vector space over $\F_q$. However, combining  Corollary 7 in \cite{LearySaad} together with the hypothesis that $\mathrm{H}_{n-1}(L^{C_{p}},\F_q)\ne 0$, shows that $\mathrm{H}_{n}(C_{H_{L}}(C_{p}),\F_q)$ is an infinite dimensional vector space over $\F_q$. In follows that $\mathrm{H}_{n+1}^{\mathcal{F}}(\Gamma,V)$ is an infinite dimensional vector space over $\F_q$ and  in particular, $\underline{\mathrm{hd}}(\Gamma)\geq n+1$. Since $\underline{\mathrm{cd}}(\G)\leq n+1$, by Theorem 6.4 of \cite{Luck1}, it follows that $\underline{\mathrm{hd}}(\G))=\underline{\mathrm{cd}}(\G)=n+1$. This concludes the case $s=1$.


For each $2\leq k\leq s$, let $\mathcal{F}_k$ be the family of finite subgroups of $\G_k=\prod_{i=1}^k H_{L_i}\rtimes C_{p_i}$. By Theorem 6.4 in \cite{Luck1}, we have that $$\underline{\mathrm{cd}}(\G_k)\leq \sum_{i=1}^k (n_i+1)=k+\sum_{i=1}^k n_i.$$
We will show that $\underline{\mathrm{hd}}(\G_k)\geq k+\sum_{i=1}^k n_i$, which will imply that $$\underline{\mathrm{hd}}(\G_k)=\underline{\mathrm{cd}}(\G_k)=k+\sum_{i=1}^k n_i.$$ To this end, define the covariant functor
\[ V_k: \mathcal{O}_{\mathcal{F}_k}\G_k \rightarrow \mathbb{Z}\mbox{-mod}: \G_k/H \mapsto 
\left\{\begin{array}{lll}  \F_q & \mbox{if}&  H=\{e\} \\
                    0            & \mbox{if}&  H\neq \{e\}, \end{array} \right. \]
where every morphism $\G_k/H \xrightarrow{g} \G_k/K$ with $K \neq \{e\}$ is mapped to the zero homomorphism, and every morphism $\G_k/\{e\} \xrightarrow{g} \G_k/\{e\}$ is mapped to the identity homomorphism. To finish the proof, it suffices to show that $\mathrm{H}^{\mathcal{F}_k}_{m_k}(\G_k,V_k)$ is an infinite dimensional  $\F_q$-vector space where $m_k=k+\sum_{i=1}^k n_i$. 

We denote $\G=H_{L_k}\rtimes C_{p_k}$, define $\mathcal{F}$ to be the family of finite subgroups of $\G$, and let $V: \mathcal{O}_{\mathcal{F}}G \rightarrow \mathbb{Z}\mbox{-mod}$ be the covariant functor as before. Note that  $G_k=\G_{k-1}\times \G $. Associated to the extension $$1 \rightarrow \G \rightarrow \G_k \xrightarrow{\pi} \G_{k-1}\rightarrow 1$$ there is a convergent $E_2$-term spectral sequence (see  \cite[5.1]{Martinez})
\begin{equation}\label{spec}
\mathrm{H}^{\mathcal{F}_{k-1}}_s(\G_{k-1},\mathrm{H}_r^{\mathcal{F}_k\cap \pi^{-1}(-)}( \pi^{-1}(-),V_k)) \Longrightarrow \mathrm{H}_{s+r}^{\mathcal{F}_k}(\G_k,V_k).
\end{equation}
Observe that for any $K\in \mathcal{F}_{k-1}$, the preimage $\pi^{-1}(K)$ is direct product of $\Gamma$ with a finite abelian group $A$. Applying the spectral sequence of \cite[5.1]{Martinez} associated to the extension $1\to A\to \pi^{-1}(K)\xrightarrow{\omega}  \G\to 1$ gives us 
\begin{eqnarray*}\mathrm{H}_r^{{\mathcal{F}_k}\cap \pi^{-1}(K)}( \pi^{-1}(K),V_k)&\cong& \mathrm{H}_r^{{\mathcal{F}_k}\cap \G}(\G,\mathrm{H}_0^{{\mathcal{F}_k}\cap \omega^{-1}(-)}(\omega^{-1}(-), V_k))\\
&\cong& \mathrm{H}_r^{{\mathcal{F}_k}\cap \G}(\G, V_k(\G_k/{\omega^{-1}(-)}))\\
&\cong& \mathrm{H}_r^{\mathcal{F}}(\G, V_k(\G_k/{\omega^{-1}(-)}))\end{eqnarray*}
where first and second isomorphisms follow from the fact that $\omega^{-1}(F)\in \mathcal{F}_k$ for every finite subgroup $F$ of $\G$. One can now easily check that  the functor $\mathrm{H}_r^{\mathcal{F}_k\cap \pi^{-1}(-)}( \pi^{-1}(-),V_k)$ is isomorphic to the functor $V_{k-1}\otimes_{\mathbb{Z}}\mathrm{H}_{r}^{\mathcal{F}}(\G,V)$.

Next, using a corner argument for the spectral sequence (\ref{spec}), we obtain \begin{eqnarray*}\mathrm{H}_{m_k}^{\mathcal{F}_k}(\G_k,V_k)&=&  \mathrm{H}_{m_{k-1}}^{\mathcal{F}_{k-1}}(\G_{k-1},\mathrm{H}_{n_k+1}^{\mathcal{F}_k\cap \pi^{-1}(-)}( \pi^{-1}(-),V_k))\\ 
&\cong&  \mathrm{H}_{m_{k-1}}^{\mathcal{F}_{k-1}}(\G_{k-1},V_{k-1}\otimes_{\mathbb{Z}}\mathrm{H}_{n_k+1}^{\mathcal{F}}(\G,V)).
\end{eqnarray*}
Now, let $P_{\ast} \rightarrow \underline{\mathbb{Z}}$ be a free contravariant $\mathcal{O}_{\mathcal{F}_{k-1}}\G_{k-1}$-resolution of $\underline{\mathbb{Z}}$. Then   $C_{\ast}=P_{\ast}\otimes_{\mathcal{F}_{k-1}}V_{k-1}$ is a complex of $\F_q$-vector spaces. One verifies that \[P_{\ast}\otimes_{\mathcal{F}_{k-1}}\Big(V_{k-1}\otimes_{\mathbb{Z}}\mathrm{H}_{n_k+1}^{\mathcal{F}}(\G,V)\Big) \cong C_{\ast} \otimes_{\mathbb{Z}} \mathrm{H}_{n_k+1}^{\mathcal{F}}(\G,V).\] This implies that
\begin{eqnarray*} \mathrm{H}_{m_{k-1}}^{\mathcal{F}_{k-1}}(\G_{k-1},V_{k-1}\otimes_{\mathbb{Z}}\mathrm{H}_{n_k+1}^{\mathcal{F}}(\G,V)) & = & \mathrm{H}_{m_{k-1}}(C_{\ast} \otimes_{\mathbb{Z}}  \mathrm{H}_{n_k+1}^{\mathcal{F}}(\G,V) )\\ &=&    \mathrm{H}_{m_{k-1}} (C_{\ast})\otimes_{\mathbb{Z}} \mathrm{H}_{n_k+1}^{\mathcal{F}}(\G,V) \\  &=&   \mathrm{H}_{m_{k-1}}^{\mathcal{F}_{k-1}}(\G_{k-1},V_{k-1})\otimes_{\mathbb{Z}}   \mathrm{H}_{n_k+1}^{\mathcal{F}}(\G,V).\end{eqnarray*}
This second equality follow from the fact that $C_{\ast}$ and $\mathrm{H}_{n_k+1}^{\mathcal{F}}(\G,V)$ are both vector spaces over $\F_q$, so tensoring with $\mathrm{H}_{n_k+1}^{\mathcal{F}}(\G,V)$ is exact. We conclude by induction that $\mathrm{H}_{m_k}^{\mathcal{F}_k}(G_k,V_k)$ is an infinite dimensional vector space over $\F_q$. This finished the proof of part (a).\\

\noindent (b). Since $\Lambda_s$ is an extension of $N=\prod_{i=1}^s H_{L_i}$ by $\Z$ it follow that
$$\mathrm{cd}(\Lambda_s)\leq 1+ \mathrm{cd}(N)= 1+\sum_{i=1}^s n_i.$$ On the other hand, since each Bestvina-Brady group $H_{L_i}$ contains a subgroup isomorphic to $\Z^{n_i}$, it is not difficult to see that $\Lambda_s$ contains a free abelian subgroups of rank $1+\sum_{i=1}^s n_i$. This shows that $1+\sum_{i=1}^s n_i\leq \mathrm{hd}(\Lambda_s)$. Since $\mathrm{hd}(\Lambda_s)\leq \mathrm{cd}(\Lambda_s)$, we deduce that 
$$\mathrm{hd}(\Lambda_s)= \mathrm{cd}(\Lambda_s)=1+\sum_{i=1}^s n_i.$$ 

Recall that right-angled Artin groups are CAT(0)-groups and hence, by Theorem 1.1 of \cite{Luck3}, have finite dimensional classifying spaces for the family of virtually cyclic subgroups. Since $N$ is an $s$-fold product of  subgroups of  right-angled Artin groups, it follows for instance from Corollary 5.6 of \cite{LuckWeiermann} that $\underline{\underline{\mathrm{gd}}}(N\times \Z) < \infty$. Because $\Lambda_s$ contains  a finite index subgroup isomorphic to $N\times \Z$, we conclude from Theorem 2.4 of \cite{Luck1} that $\underline{\underline{\mathrm{gd}}}(\Lambda_s) < \infty$. Hence, $\underline{\underline{\mathrm{cd}}}(\Lambda_s) < \infty$.

\indent By part (a), we have $\mathrm{cd}(\G_s)=s+\mathrm{cd}(N)$ with $s\geq 2$. So,
Lemma \ref{lem: delta(N)} below gives us  
$$\underline{\underline{\mathrm{cd}}}(\Lambda_s)\geq s+ {{\mathrm{cd}}}(N)= s+\sum_{i=1}^s n_i.$$

Finally, since $\Lambda_s$ is a torsion-free finite extension of $N\times \Z$ which is an integral linear group of type $FP$, it  follows that $\Lambda_s$ is an integral linear group of type $FP$. If additionally $L$ is contractible, then $\Lambda_s$ is of type $F$.
\end{proof}
\begin{lemma}\label{lem: delta(N)} Let $N \to \G \xrightarrow{\pi} C_d$ be a short exact sequence of  groups such that $\underline{\mathrm{cd}}(\G)\geq\underline{\mathrm{cd}}(N)+s$ for some $s\geq 2$. 
Let $\Lambda=N\rtimes \Z$ be the pull-back of $\G$ under projection of 
$\Z$ onto $C_d$. Then  $\underline{\underline{\mathrm{cd}}}(\Lambda)\geq\underline{\mathrm{cd}}(N)+s$.
\end{lemma}
\begin{proof} 
The group $\Lambda$ fits into the following commutative diagram with exact rows and surjective vertical maps
\[   \xymatrix{  1 \ar[r] & N \ar[r] \ar[d]^{\mathrm{Id}} & \Lambda \ar[r] \ar[d] & \mathbb{Z} \ar[d] \ar[r]  & 1 \\
                1 \ar[r] & N \ar[r]  & \G \ar[r]^{\pi}  & {C}_d  \ar[r]  & 1. } \]
Note that $\Lambda$ has a infinite cyclic normal subgroup $H$ such that $\Lambda/H=\mathrm{Comm}_{\Lambda}[H]/H=\G$. Therefore, it follows from Lemma \ref{lemma: virt poly} that $$\mathrm{cd}_{\mathcal{F}[H]}(\Lambda)=\mathrm{cd}_{\mathcal{F}[H]}(\mathrm{Comm}_{\Lambda}[H])=\underline{\mathrm{cd}}(\G)\geq \underline{\mathrm{cd}}(N)+s.$$

\noindent Using Lemma \ref{lemma: fam[H]}, we obtain 
\begin{eqnarray*}
\underline{\mathrm{cd}}(N)+s&\leq& \mathrm{cd}_{\mathcal{F}[H]}(\Lambda)\\
 &\leq& \max\{\underline{\underline{\mathrm{cd}}}(\Lambda),\underline{\mathrm{cd}}(\Lambda)\}\\
&\leq& \max\{\underline{\underline{\mathrm{cd}}}(\Lambda),\underline{\mathrm{cd}}(N)+1\},
\end{eqnarray*}
which implies that $\underline{\underline{\mathrm{cd}}}(\Lambda)\geq \underline{\mathrm{cd}}(N)+s$. 
\end{proof}

\begin{example}\label{ex: counterex} \rm \indent In  \cite{Martinez12}, Mart\'{\i}nez-P\'{e}rez constructs, for any odd prime $p$, a Bestvina-Brady group $H_{L_p}$ of type $F$ and a semi-direct product $\G=H_{L_p} \rtimes C_p$ such that $\mathrm{cd}(H_{L_p})=3$ but $\underline{\mathrm{cd}}(\G)=4$. To put this into the context of Theorem C, we give a few details of this  construction as presented in Example 3.6 of \cite{Martinez12}.

Consider the real projective plane $\mathbb{R}P^2$ with the standard CW-structure consisting of one cell in each dimension and let $p$ be an odd prime. By a theorem of Jones (see \cite[Theorem 1.1]{Jones}), $\mathbb{R}P^2$ can be embedded in a contractible $3$-dimensional finite $C_p$-CW-complex $X$ such that $X^{C_p}=\mathbb{R}P^2$. Let $L$ be the flag triangulation of $X$ and $H_{L}$ the associated Bestvina-Brady group. Set $\G=H_{L}\rtimes C_p$ where the action of $C_p$ on $H_{L}$ is induced by the action of $C_p$ on the flag complex $L$. Mart\'{\i}nez-P\'{e}rez then obtains  the cohomological finiteness conditions for $\G$ stated above.

Now, we let $\{p_i \; | \; i=1, \dots , s\}$ be a collection of distinct odd primes and for each $i$,  denote by  $L_i$  the flag complex associated to the prime $p_i$ in the  construction of Mart\'{\i}nez-P\'{e}rez. Note that the hypothesis of Theorem C are satisfied for $q=2$ because $n_i=3$ and $L^{C_{p_i}}=\mathbb{R}P^2$ for all $i$. Therefore, we can conclude that for each $s \in \mathbb{N}$,  the group $N_s= \prod_{i=1}^s H_{L_i}$ is an integral linear group of type $F$ and there is a semi-direct product $\G_s=N_s \rtimes C_d$,  where $d=\prod_{i=1}^s p_i$  such that $$\mathrm{hd}(N_s)=\mathrm{cd}(N_s)=3s \;\;\; \mbox{ and } \;\;\; \underline{\mathrm{hd}}(\G_s)=\underline{\mathrm{cd}}(\G_s)=4s.$$
Note that this implies \[\delta(N_s)-\underline{\mathrm{gd}}(N_s)\geq s\] for every $s \in \mathbb{N}$. Moreover, taking any $s\geq 2$, we obtain $$\mathrm{gd}(\Lambda_s)+s-1\leq \underline{\underline{\mathrm{gd}}}(\Lambda_s)<\infty $$
where $\Lambda_s=N_s\rtimes \Z$ an integral linear group of type $F$ formed by taking the pull-back of $\G_s$ under the projection of $\Z$ onto $C_d$. 
\end{example}

\begin{remark}\rm \label{remk: Question of Lafont}  In \cite[46.4]{Lafont}  (see also \cite{guido}), Lafont asked which finitely generated groups $\G$ have finite ${\underline{\mathrm{gd}}}(\G)$ but infinite $\underline{\underline{\mathrm{gd}}}(\G)$. Next, we explain how the invariant $\delta(N)$ is connected to this question.

Suppose there exists a  countable group $N$ such that $\underline{\mathrm{gd}}(N)<\infty$ but  $\delta(N)=\infty$. This means that there is an infinite sequence of subgroups $\{K_i\}$ of $N$ and a collection of groups $\{\Gamma_i\}$   where $\G_i$ is an extension of $K_i$ by a finite cyclic group $C_{d_i}$ such that 
 $$i\leq\underline{\mathrm{cd}}(\G_i)<\infty,$$
 for each $i\in \mathbb N$. Let $\Lambda_i=K_i\rtimes \Z$ be the pull-back of $\G_i$ under projection of $\Z$ onto $C_{d_i}$ and let $H_i$ be the infinite cyclic central subgroup of $\Lambda_i$ such that $\Lambda_i/H_i= \G_i$ for each $i\in \mathbb N$.
Proceeding as in the proof of Lemma \ref{lem: delta(N)}, we obtain $\sup_{i \in \mathbb{N}}\{\mathrm{cd}_{\mathcal{F}[H_i]}(\Lambda_i)\}=\infty$. Since $\underline{\mathrm{cd}}(\Lambda_i)\leq \underline{\mathrm{cd}}(N)+1$  for each $i$, by Lemma \ref{lemma: fam[H]}, we have  that  $\sup_{i \in \mathbb{N}}\{\underline{\underline{\mathrm{cd}}}(\Lambda_i)\}=\infty$.

Define $\Lambda$ to be the infinite free product $*_{i\in \mathbb N} \Lambda_i$. Clearly,  $\underline{\underline{\mathrm{gd}}}(\Lambda)=\infty$. Since $\Lambda$ acts on a tree  with each vertex stabilizer conjugate to  $\Lambda_i$ for some $i\in \mathbb N$, using Corollary 4.1 of \cite{DemPetTal}, it is not difficult to check that ${\underline{\mathrm{cd}}}(\Lambda)\leq {\underline{\mathrm{cd}}}(N)+2$. In the next proposition we will show that $\Lambda$ can be embedded into a 2-generated group with the same Bredon cohomological dimension for the family of finite subgroups. 

To summarize, we have just shown that if there exists a  countable group $N$ such that $\underline{\mathrm{gd}}(N)<\infty$ but $\delta(N)=\infty$. Then, there exists a 2-generated group $\G$ such that $\underline{\mathrm{gd}}(\G)<\infty$ but $\underline{\underline{\mathrm{gd}}}(\G)=\infty$.
\end{remark}

\begin{proposition} Let $\Gamma$ be a countable group. Then $\Gamma$ can be embedded into a 2-generated group $H$ such that 
$$\underline{\mathrm{cd}}(\Gamma)\leq \underline{\mathrm{cd}}(H)\leq \mathrm{max}\{\underline{\mathrm{cd}}(\Gamma), 2\}.$$
\end{proposition}
\begin{proof} Following the  proof of Higman-Neumann-Neumann in \cite[Theorem IV]{HNN}, note that first, $\Gamma$ is  embedded in a group $K$ which is a free product of $\Gamma$ and $\Z$.  The Mayer-Vietoris sequence associated to the push-out construction of  a model for $\underline{E}K$ in Example 4.10 of \cite{Luck2} then implies that $$\underline{\mbox{cd}}(K)\leq \mbox{max}\{\underline{\mbox{cd}}(\Gamma), \underline{\mbox{cd}}(\Z)\}=\underline{\mbox{cd}}(\Gamma).$$ 

Next, the group $P$ is defined as the fundamental group of a graph of groups which is a (possibly infinite) wedge of circles with $K$ as its vertex group and the edge subgroups  all isomorphic to $\Z$. Similarly, using a generalization of the construction in \cite[4.11]{Luck2} to a wedge of circles, one can deduce that $\underline{\mbox{cd}}(P)\leq \max\{\underline{\mbox{cd}}(K),2\}$.
 
 Now, the group $Q$ is defined as an amalgamated free product of $P$ and a free group on 2-generators. Again,  by \cite[4.10]{Luck2} and the fact that non-abelian free groups have cohomological dimension 1, it follows that 
 $$\underline{\mbox{cd}}(Q)\leq \mbox{max}\{\underline{\mbox{cd}}(P), 2\}.$$ 
 In \cite{HNN}, it is also shown that $Q$ is generated by the elements $u$, $v$, and $b$ such that both $u$, $b$ and $b$,  $v$ generate free subgroups of $Q$. Finally, the group $H$ defined as 
$$\langle Q, a \;|\; a^{-1}ba=u, a^{-1}va=b\rangle$$ 
is $2$-generated, and can be seen as an HNN-extension of $Q$  relative to the isomorphism $$\theta: \langle b, v\rangle \to \langle u, b\rangle: \;\; b\mapsto u, \; v\mapsto b.$$ Hence, by \cite[4.11]{Luck2}, we have $$\underline{\mbox{cd}}(H)\leq \mbox{max}\{\underline{\mbox{cd}}(Q), 2\}.$$
The  result now follows from combining the above inequalities.
\end{proof}

\section{The Mayer-Vietoris sequence of L\"{u}ck and Weiermann}
Let $\mathcal{F} \subseteq \mathcal{H}$ be two families of subgroups of a group $G$ such that the set $\mathcal{S}=\mathcal{H}\smallsetminus \mathcal{F}$ is equipped with a strong equivalence relation.
Let $[H]$ be the equivalence class represented by $H \in \mathcal{S}$, denote the set of equivalence classes by $[\mathcal{S}]$ and let $\mathcal{I}$ be a complete set of representatives $[H]$ of the orbits of the conjugation action of $G$ on $[\mathcal{S}]$. L\"{u}ck and Weiermann obtained the following Mayer-Vietoris long exact sequence from the push-out diagram of Theorem $2.3.$ in \cite{LuckWeiermann}.
\begin{proposition}[L\"{u}ck-Weiermann, \cite{LuckWeiermann}] \label{prop: mayer-vietoris} Let $M \in \mbox{Mod-}\mathcal{O}_{\mathcal{H}}G$. There exists a long exact cohomology sequence
\[ \ldots \rightarrow \mathrm{H}^{i}_{\mathcal{H}}(G,M) \rightarrow \Big(\prod_{[H] \in \mathcal{I}} \mathrm{H}^{i}_{\mathcal{F}[H] }( \mathrm{N}_{G}[H],M)\Big)\oplus  \mathrm{H}^{i}_{\mathcal{F}}(G,M) \rightarrow  \prod_{[H] \in \mathcal{I}} \mathrm{H}^{i}_{\mathcal{F}\cap  \mathrm{N}_{G}[H] }( \mathrm{N}_{G}[H],M)  \]
\[ \rightarrow \mathrm{H}^{i+1}_{\mathcal{H}}(G,M) \rightarrow \ldots \ . \]
\end{proposition}
A consequence of this long exact sequence is the following algebraic analogue of Theorem \ref{th: push out}.
\begin{theorem} \label{th: cd push out} Let $\mathcal{F} \subseteq \mathcal{H}$ be two families of subgroups of a group $G$ such that $\mathcal{S}=\mathcal{H}\smallsetminus \mathcal{F}$ is equipped with a strong equivalence relation. Denote the set of equivalence classes by $[\mathcal{S}]$ and let $\mathcal{I}$ be a complete set of representatives $[H]$ of the orbits of the conjugation action of $G$ on $[\mathcal{S}]$. If there exists a natural number $k$ such that for each $[H] \in \mathcal{I}$
\smallskip
\begin{itemize}
\item[-] $\mathrm{cd}_{\mathcal{F}\cap \mathrm{N}_{G}[H]}(\mathrm{N}_{G}[H]) \leq k-1$;
\medskip
\item[-] $\mathrm{cd}_{\mathcal{F}[H]}(\mathrm{N}_{G}[H]) \leq k$,
\end{itemize}
and such that $\mathrm{cd}_{\mathcal{F}}(G) \leq k$, then $\mathrm{cd}_{\mathcal{H}}(G) \leq k$.
\end{theorem}

Our goal in this section is to give an algebraic construction of the long exact sequence of Proposition \ref{prop: mayer-vietoris}. The added value of this approach is that our proof will produce long exact sequences for arbitrary Tor and Ext functors.

We start by recalling some basic facts about the (co)homology of an orbit category. We refer the reader to \cite[section 9]{Luck}, \cite{FluchThesis} and \cite{nucinkis} for more details.

Let  $\orb$ be the orbit category. In  $\orb$, every morphism $\varphi: G/H \rightarrow G/K$ is completely determined by $\varphi(H)$, since $\varphi(xH)=x\varphi(H)$ for all $x \in G$. Moreover, there exists a morphism \[G/H \rightarrow G/K : H \mapsto xK\] if and only if $x^{\scriptscriptstyle -1}Hx \subseteq K$. We denote the morphism \[\varphi: G/H \rightarrow G/K: H\mapsto xK\]  by $G/H \xrightarrow{x} G/K$. The set of morphisms from $G/H$ to $G/K$ is denoted by $\mathrm{Mor}(G/H,G/K)$.\\
\indent A \emph{right $\orb$-module} is a contravariant functor $M: \orb \rightarrow  \mathbb{Z}\mbox{-mod}$. The \emph{category of right $\orb$-modules} is denoted by $\orbmod$ and is defined by the objects that are all the right $\orb$-modules and the morphisms are all the natural transformations between the objects. A sequence \[0\rightarrow M' \rightarrow M \rightarrow M'' \rightarrow 0\]
in $\orbmod$ is called {\it exact} if it is exact after evaluating in $G/H$, for all $H \in \mathcal{F}$. Take $M \in \orbmod$ and consider the left exact functor:
\[ \nathom(M,-) : \orbmod \rightarrow  \mathbb{Z}\mbox{-mod}: N \mapsto \nathom(M,N), \]
where $\nathom(M,N)$ is the abelian group of all natural transformations from $M$ to $N$. By definition, $M$ is a projective $\orb$-module if and only if this functor is exact. It can be shown that $\orbmod$ is an abelian category that contains enough projective objects to construct projective resolutions. Hence, one can construct functors $\mathrm{Ext}^{\ast}_{\orb}(-,M)$ as the right derived functors of the contravariant right exact functor \[\nathom(-,M): \orbmod \rightarrow  \mathbb{Z}\mbox{-mod}: N \mapsto \nathom(N,M).\] These Ext functors have all the usual properties, and one has $\mathrm{H}^{\ast}_{\mathcal{F}}(G,M)=\mathrm{Ext}^{\ast}_{\orb}(\underline{\mathbb{Z}},M)$ by definition.

Similarly, the \emph{category of left $\orb$-modules} is denoted by $\mathcal{O}_{\mF}G\mbox{-Mod}$ and is defined by the objects that are all covariant functors from $\orb$ to $ \mathbb{Z}\mbox{-mod}$  and the morphisms are all the natural transformations between the objects. \\
\indent For any $L \in \mathcal{F}$, one can define the right $\orb$-module
\[ \mathbb{Z}[-,G/L]: \orb \rightarrow  \mathbb{Z}\mbox{-mod} : G/H \rightarrow \mathbb{Z}[G/H,G/L], \]
where $\mathbb{Z}[G/H,G/L]$ is the free abelian group generated by $\mathrm{Mor}(G/H,G/L)$. Here, a morphism $\varphi \in \mathrm{Mor}(G/H_1,G/H_2)$ is mapped to $$\mathbb{Z}[G/H_2,G/L] \rightarrow \mathbb{Z}[G/H_1,G/L]: \alpha \mapsto \alpha \circ \varphi.$$ Similarly, one can define the left $\orb$-module $\mathbb{Z}[G/L,-]$ for every  $L \in \mathcal{F}$.

Given $M \in \orbmod$ and $N \in \mathcal{O}_{\mF}G\mbox{-Mod}$, one defines their tensor product $M \otimes_{\orb}N$ to be the following abelian group
\[ M \otimes_{\orb}N = \bigoplus_{S \in \mathcal{F}}M(G/S)\otimes N(G/S) \big/ I, \]
where $I$ is the abelian group generated by the elements $M(\varphi)(m)\otimes n - m\otimes N(\varphi)(n)$, for all $\varphi \in \mathrm{Mor}(G/H,G/K)$, $m \in M(G/K)$, $n \in N(G/H)$ and for all $H,K \in \mathcal{F}$. We view elements of  $M \otimes_{\orb}N$ as finite sum of elementary tensors of the form $m\otimes n \in M(G/H)\otimes N(G/H)$ for all $H \in \mathcal{F}$, that satisfy $M(\varphi)(m)\otimes n = m\otimes N(\varphi)(n)$ for all $\varphi \in \mathrm{Mor}(G/H,G/K)$. When the group is clear from the context, we will denote $M \otimes_{\orb}N$ by $M \otimes_{\mathcal{F}}N$. \\
\indent Given $N \in \mathcal{O}_{\mF}G\mbox{-Mod}$, define the right exact functor
\[ -\otimes_{\orb}N :  \orbmod \rightarrow \mathbb{Z}\mbox{-mod}: M \rightarrow M \otimes_{\orb}N. \]
The functors $\mathrm{Tor}_{\ast}^{\orb}(-,N)$ are the left derived functors of this right exact functor. These Tor functors can be computed using projective resolutions, or more generally, using flat resolutions. The \emph{$n$-th Bredon homology} of the group $G$ with coefficients in $N \in  \mathcal{O}_{\mF}G\mbox{-Mod}$ is by definition \[\mathrm{H}_{\ast}^{\mathcal{F}}(G,N)=\mathrm{Tor}_{\ast}^{\orb}(\underline{\mathbb{Z}},N).\] There is a notion of \emph{Bredon homological dimension} of the group $G$ for the family $\mathcal{F}$, which is defined as the integer
\[      \mathrm{hd}_{\mathcal{F}}(G)=\sup\{    n \in \mathbb{N} \ | \ \exists N \in   \mathcal{O}_{\mF}G\mbox{-Mod}: \mathrm{H}_{n}^{\mathcal{F}}(G,N) \neq 0 \}.           \] 
In general, the inequality $\mathrm{hd}_{\mathcal{F}}(G)\leq \mathrm{cd}_{\mathcal{F}}(G)$ always holds. For countable groups $G$, one even has
\[\mathrm{hd}_{\mathcal{F}}(G)\leq \mathrm{cd}_{\mathcal{F}}(G)\leq \mathrm{hd}_{\mathcal{F}}(G)+1.\]
\indent Let $K$ be a group and let $\mathcal{V}$ be a family of subgroups of $K$. Associated to a functor $\pi: \mathcal{O}_{\mathcal{V}}K \rightarrow \orb$, one has a restriction functor
\[ \mathrm{res}_{\pi}: \orbmod \rightarrow \mbox{Mod-}\mathcal{O}_{\mathcal{V}}K : M \mapsto M \circ \pi \]
and an induction functor
\[ \mathrm{ind}_{\pi}: \mbox{Mod-}\mathcal{O}_{\mathcal{V}}K \rightarrow \orbmod: M \mapsto M(?) \otimes_{\mathcal{O}_{\mathcal{V}}K} \mathbb{Z}[-,\pi(?)]. \]
The functor $\mathrm{ind}_{\pi}$ is left adjoint to the functor $\mathrm{res}_{\pi}$. In particular, there is an isomorphism
\begin{equation} \label{eq: adjoint} \mathrm{Hom}_{\mathcal{O}_{\mathcal{F}}G}(\mathrm{ind}_{\pi}(M),N)\cong \mathrm{Hom}_{\mathcal{O}_{\mathcal{V}}K}(M,\mathrm{res}_{\pi}(N)) \end{equation}
that is natural in $M \in  \mbox{Mod-}\mathcal{O}_{\mathcal{V}}K$ and $N \in \orbmod$. 

\begin{conv} From now on, restriction functors will be omitted from notation.
\end{conv}
\indent Before we proceed, let us establish some notation. The induction functor associated to the functor
\[ \orb \rightarrow \mathcal{O}_{\mathcal{H}}G:  G/F \mapsto G/F \]
will be denoted by $\mathrm{ind}_{\mathcal{F}}^{\mathcal{H}}$.
The induction functor associated to the functor
\[ \mathcal{O}_{\mathcal{F}\cap  \mathrm{N}_{G}[H]} \mathrm{N}_{G}[H] \rightarrow  \mathcal{O}_{\mathcal{H}}G : \mathrm{N}_{G}[H]/F \mapsto G/F \]
will be denoted by $\mathrm{ind}_{\mathcal{F}\cap  \mathrm{N}_{G}[H]}^{\mathcal{H}}$. And finally, the induction functor associated to the functor
\[ \mathcal{O}_{\mathcal{F}[H]} \mathrm{N}_{G}[H] \rightarrow \mathcal{O}_{\mathcal{H}}G : \mathrm{N}_{G}[H]/F \mapsto G/F  \]
will be denoted by $\mathrm{ind}_{\mathcal{F}[H]}^{\mathcal{H}}$.

In the next three lemmas, we will prove some elementary properties of the induction functors introduced above.
\begin{lemma} \label{lemma: induction} Let $G$ be a group and let $\mathcal{F} \subseteq \mathcal{H}$ be an inclusion of families of subgroups of $G$. If $M \in \orbmod$, then $M$ and $\mathrm{ind}_{\mathcal{F}}^{\mathcal{H}}(M)$ are isomorphic as $\orb$-modules and $\mathrm{ind}_{\mathcal{F}}^{\mathcal{H}}(M)(G/S)=0$ for all $S \in \mathcal{H}\smallsetminus \mathcal{F}$.
\end{lemma}
\begin{proof}Let $H \in \mathcal{F}$.
Note that in $M\otimes_{\mathcal{O}_{\mathcal{F}}G} \mathbb{Z}[G/H,-]$, we have $m\otimes \varphi= M(\varphi)(m)\otimes \mathrm{Id}_H$ for any $m \in M(G/K)$ and $\varphi \in \mathrm{Mor}(G/H,G/K)$. Here $\mathrm{Id}_H$ denotes the identity morphism $G/H \rightarrow G/H$. Using this observation, one checks that the map \[\eta(H): M\otimes_{\mathcal{O}_{\mathcal{F}}G} \mathbb{Z}[G/H,-] \rightarrow M(G/H),\] that sends $m\otimes \varphi \in M(G/K)\otimes \mathbb{Z}[G/H,G/K]$ to $M(\varphi)(m)\in M(G/H)$, is a well-defined isomorphism. Moreover, it is easily seen that the maps $\eta(H)$ assemble to form a natural transformation $\eta : \mathrm{ind}_{\mathcal{F}}^{\mathcal{H}}(M) \rightarrow M$ of $\orb$-modules. This proves the first statement of the lemma. \\
\indent Now, let $S \in \mathcal{H}\smallsetminus \mathcal{F}$. By definition, we have $$\mathrm{ind}_{\mathcal{F}}^{\mathcal{H}}(M)(G/S)= M(-) \otimes_{\mathcal{F}} \mathbb{Z}[G/S,-].$$ If there existed a subgroup $F \in \mathcal{F}$ and a morphism $G/S \xrightarrow{x} G/F$, then this would imply that $S \in \mathcal{F}$ because $\mathcal{F}$ is subgroup and conjugation closed. We conclude that $M(-) \otimes_{\mathcal{F}} \mathbb{Z}[G/S,-]=0$. This completes the proof.
\end{proof}
\begin{lemma} \label{lemma: injective} Let $M \in \mbox{Mod-}\mathcal{O}_{\mathcal{H}}G$, $H \in \mathcal{S}$ and $K \in \mathcal{F}$. The natural map
\[ \mathrm{ind}_{\mathcal{F}\cap  \mathrm{N}_{G}[H]}^{\mathcal{H}}(M)(G/K) \rightarrow  \mathrm{ind}_{\mathcal{F}[H]}^{\mathcal{H}}(M)(G/K): m\otimes \varphi \mapsto m \otimes \varphi\]
is an isomorphism.
\end{lemma}
\begin{proof}
Let $L \in \mathcal{F}[H]$, $m \in M(G/L)$ and (when possible) $\varphi : G/K \xrightarrow{x} G/L$ in $\mathbb{Z}[G/K,G/L]$. We have \[\Big(G/K \xrightarrow{x} G/L\Big)=\Big(  G/K \xrightarrow{x} G/K^{x} \rightarrow G/L \Big)\] and $K^{x} \in \mathcal{F}\cap  \mathrm{N}_{G}[H]$. This implies that
 \[m\otimes \varphi = M( G/K^{x} \rightarrow G/L )(m)\otimes (G/K \xrightarrow{x} G/K^{x}) \in M\otimes_{\mathcal{F}[H]}\mathbb{Z}[G/K,-].\]
This shows that every element of $M \otimes_{\mathcal{F}[H]}\mathbb{Z}[G/K,-]$ is a sum of elementary tensors of the form $n\otimes (G/K \xrightarrow{x} G/K^{x})$ for some $x \in G$ and some $n \in M(G/K^{x})$. Using this fact, it is not difficult to check that the map under consideration is an isomorphism.
\end{proof}
\begin{lemma} \label{lemma: exact and project}The induction functors  $\mathrm{ind}_{\mathcal{F}}^{\mathcal{H}}$, $\mathrm{ind}_{\mathcal{F}\cap  \mathrm{N}_{G}[H]}^{\mathcal{H}}$ and $\mathrm{ind}_{\mathcal{F}[H]}^{\mathcal{H}}$ are exact and preserve projectives.
\end{lemma}
\begin{proof} All induction functors preserve projectives since they are left adjoint to a restriction functor which is exact.\\
\indent Let $\Gamma$ be a group and let $\mathcal{F} \subseteq \mathcal{H}$ be two families of subgroups of $\Gamma$. Then the induction functor associated to the functor \[ \mathcal{O}_{\mathcal{F}}\Gamma \rightarrow \mathcal{O}_{\mathcal{H}}\Gamma:  \Gamma/F \mapsto \Gamma/F\] is exact, by Lemma \ref{lemma: induction}. In particular, $\mathrm{ind}_{\mathcal{F}}^{\mathcal{H}}$ is exact. Now, let $\Gamma$ be a group, $K$ be a subgroup of $\Gamma$ and let $\mathcal{F}$ be a family of subgroups of $\Gamma$. Lemma 2.9 in \cite{Symonds} (or see Proposition 3.26 in \cite{FluchThesis}) shows that the induction functor associated to the functor $$ \mathcal{O}_{\mathcal{F}\cap K}K \rightarrow \mathcal{O}_{\mathcal{F}}\Gamma:  K/F \mapsto \Gamma/F$$ is exact. Since the functors $\mathrm{ind}_{\mathcal{F}\cap  \mathrm{N}_{G}[H]}^{\mathcal{H}}$ and $\mathrm{ind}_{\mathcal{F}[H]}^{\mathcal{H}}$ can be written as a composition of the two previous induction functors, we conclude that they are exact as well.
\end{proof}
The following proposition provides the main algebraic tool needed to construct the long exact Mayer-Vietoris sequence.
\begin{proposition}\label{prop: short exact seq} Let $M \in \mbox{Mod-}\mathcal{O}_{\mathcal{H}}G$. We have a short exact sequence
\[ 0 \rightarrow \bigoplus_{[H] \in \mathcal{I}}\mathrm{ind}_{\mathcal{F}\cap  \mathrm{N}_{G}[H]}^{\mathcal{H}}(M) \rightarrow  \Big(\bigoplus_{[H] \in \mathcal{I}}\mathrm{ind}_{\mathcal{F}[H]}^{\mathcal{H}}(M)\Big) \oplus \mathrm{ind}_{\mathcal{F}}^{\mathcal{H}}(M) \rightarrow M \rightarrow 0. \]
of $\mathcal{O}_{\mathcal{H}}G$-modules.
\end{proposition}
\begin{proof}
First, notice that \[\mathrm{ind}_{\mathcal{F}\cap  \mathrm{N}_{G}[H]}^{\mathcal{H}}(M)(K)= \mathrm{ind}_{\mathcal{F}}^{\mathcal{H}}(M)(K)=0\] if $K \in \mathcal{S}$ and that $\mathrm{ind}_{\mathcal{F}}^{\mathcal{H}}(M)(K)=M(K)$ if $K \in \mathcal{F}$, by Lemma \ref{lemma: induction}. Suppose $K \in \mathcal{F}$.
Let us first define the map
\[ i(K): \bigoplus_{[H] \in \mathcal{I}}\mathrm{ind}_{\mathcal{F}\cap  \mathrm{N}_{G}[H]}^{\mathcal{H}}(M)(K) \rightarrow  \Big(\bigoplus_{[H] \in \mathcal{I}}\mathrm{ind}_{\mathcal{F}[H]}^{\mathcal{H}}(M)(K)\Big) \oplus M(G/K). \]
Let \[ m\otimes \varphi \in \mathrm{ind}_{\mathcal{F}\cap  \mathrm{N}_{G}[H]}^{\mathcal{H}}(M)(K)= M \otimes_{\mathcal{F}\cap  \mathrm{N}_{G}[H]}\mathbb{Z}[G/K,-] \] and define \[i(K)(m\otimes \varphi)=(m\otimes \varphi,-M(\varphi)(m)) \in  \mathrm{ind}_{\mathcal{F}[H]}^{\mathcal{H}}(M)(K) \oplus M(G/K)  .\] If $K \in \mathcal{S}$, then set $i(K)= 0$. One can now easily check that the maps $i(K)$, for $K \in \mathcal{H}$, assemble to form a natural transformation
\[ i: \bigoplus_{[H] \in \mathcal{I}}\mathrm{ind}_{\mathcal{F}\cap  \mathrm{N}_{G}[H]}^{\mathcal{H}}(M) \rightarrow  \Big(\bigoplus_{[H] \in \mathcal{I}}\mathrm{ind}_{\mathcal{F}[H]}^{\mathcal{H}}(M)(K)\Big) \oplus  \mathrm{ind}_{\mathcal{F}}^{\mathcal{H}}(M) . \]
Let $K \in \mathcal{F}$ and define the map
\[ p(K):  \Big(\bigoplus_{[H] \in \mathcal{I}}\mathrm{ind}_{\mathcal{F}[H]}^{\mathcal{H}}(M)(G/K)\Big) \oplus M(G/K) \rightarrow M(G/K): (m\otimes \varphi, n)\mapsto M(\varphi)(m)+n.\]
If $K \in \mathcal{S}$, then define
\[ p(K):  \Big(\bigoplus_{[H] \in \mathcal{I}}\mathrm{ind}_{\mathcal{F}[H]}^{\mathcal{H}}(M)(G/K)\Big) \rightarrow M(G/K) :  m\otimes \varphi \mapsto M(\varphi)(m).\]
One can check that the maps $p(K)$ assemble to form a natural transformation
\[ p:  \Big(\bigoplus_{[H] \in \mathcal{I}}\mathrm{ind}_{\mathcal{F}[H]}^{\mathcal{H}}(M)(G/K)\Big) \oplus  \mathrm{ind}_{\mathcal{F}}^{\mathcal{H}}(M)  \rightarrow M.\]
\indent We claim that the sequence
\[ 0 \rightarrow \bigoplus_{[H] \in \mathcal{I}}\mathrm{ind}_{\mathcal{F}\cap  \mathrm{N}_{G}[H]}^{\mathcal{H}}(M) \xrightarrow{i}  \Big(\bigoplus_{[H] \in \mathcal{I}}\mathrm{ind}_{\mathcal{F}[H]}^{\mathcal{H}}(M)\Big) \oplus \mathrm{ind}_{\mathcal{F}}^{\mathcal{H}}(M) \xrightarrow{p} M \rightarrow 0 \]
is exact. To prove this, first let $K \in \mathcal{F}$. Then $p(K)$ is obviously surjective and the injectivity of $i(K)$ follows from Lemma \ref{lemma: injective}.
Clearly, we also have $p(K)\circ i(K)=0$. If an element of \[\Big(\bigoplus_{[H] \in \mathcal{I}}\mathrm{ind}_{\mathcal{F}[H]}^{\mathcal{H}}(M)(K)\Big)\oplus M(G/K)\] is in the kernel of $p(K)$, then Lemma \ref{lemma: injective} implies that this element is in the image if $i(K)$. Now let $K \in \mathcal{S}$. We need to show that
\begin{equation} \label{eq: p} p(K):  \bigoplus_{[H] \in \mathcal{I}}M\otimes_{\mathcal{F}[H]}\mathbb{Z}[G/K,-] \rightarrow  M(G/K) \end{equation}
is an isomorphism. Consider $[H] \in \mathcal{I}$ and assume that there exists an $L \in \mathcal{F}[H]$ such that $\mathbb{Z}[G/K,G/L]$ is non-zero. Then there exists an $x \in G$ such that $K^{x} \subseteq L$. Since $L \sim H$, it follows that $K \sim {}^{x}H$. By definition of $\mathcal{I}$, there can be at most one $[H] \in \mathcal{I}$ such that $\mathbb{Z}[G/K,G/L]$ is non-zero for some $L \in \mathcal{F}[H]$. On the other hand, there must exist an $[S] \in \mathcal{I}$ such that $K^{x} \sim S$. This implies that $K^{x} \in \mathcal{F}[S]$, so $\mathbb{Z}[G/K,G/K^{x}]$ is non-zero. We conclude that there exists a unique $[S] \in \mathcal{I}$ such that (\ref{eq: p}) transforms to
\[  p(K):  M\otimes_{\mathcal{F}[S]}\mathbb{Z}[G/K,-] \rightarrow  M(G/K): m\otimes \varphi \mapsto M(\varphi)(m), \]
and such that $K^{x} \in \mathcal{F}[S]$ for some $x \in G$.
Denote the morphism $G/K \xrightarrow{x} G/K^{x}$ by $\varphi_0$.
Let $L \in \mathcal{F}[S]\smallsetminus \mathcal{F}\cap N_{G}[S]$ and consider (when possible) a morphism $\varphi: G/K \xrightarrow{y} G/L$. The existence of this morphism implies that $K \sim {}^{y}S$. Since we also have $K \sim {}^{x}S$, it follows that $x^{-1}y \in \mathrm{N}_{G}[S]$. By noting that $$\Big(G/K \xrightarrow{y} G/L\Big)=\Big(  G/K \xrightarrow{x} G/K^{x} \xrightarrow{x^{-1}y} G/L \Big) ,$$ we see that \[m\otimes \varphi = M( G/K^{x} \xrightarrow{x^{-1}y} G/L )(m)\otimes \varphi_0 \in M\otimes_{\mathcal{F}[S]}\mathbb{Z}[G/K,-],\] for any $m \in M(G/L)$. This shows that every element in $M\otimes_{\mathcal{F}[S]}\mathbb{Z}[G/K,-]$ is of the form $m \otimes \varphi_0$ for some $m \in M(G/K^{x})$.
Since the morphism $\varphi_0: G/K \xrightarrow{x} G/K^{x} $ has a two-sided inverse, the map $M(\varphi_0)$ is an isomorphism. It follows that $p(K)$ is an isomorphism.
\end{proof}
We can now prove the main result of this section.
\begin{theorem} \label{th: ext} Let $\mathcal{F} \subseteq \mathcal{H}$ be two families of subgroups of a group $G$ such that the set $\mathcal{S}=\mathcal{H}\smallsetminus \mathcal{F}$ is equipped with a strong equivalence relation. Let $[H]$ be the equivalence class represented by $H \in \mathcal{S}$, denote the set of equivalence classes by $[\mathcal{S}]$ and let $\mathcal{I}$ be a complete set of representatives $[H]$ of the orbits of the conjugation action of $G$ on $[\mathcal{S}]$. Take $M,N \in \mbox{Mod-}\mathcal{O}_{\mathcal{H}}G$ and $T \in \mathcal{O}_{\mathcal{H}}G \mbox{-Mod}$. Then, there exists a long exact sequence of Ext functors
\[ \ldots \rightarrow \mathrm{Ext}^{i}_{\mathcal{O}_{\mathcal{H}}G}(M,N) \rightarrow \Big(\prod_{[H] \in \mathcal{I}} \mathrm{Ext}^{i}_{ \mathcal{O}_{\mathcal{F}[H]}\mathrm{N}_{G}[H]}(M,N)\Big)\oplus  \mathrm{Ext}^{i}_{\mathcal{O}_{\mathcal{F}}G}(M,N) \rightarrow   \]
\[ \prod_{[H] \in \mathcal{I}} \mathrm{Ext}^{i}_{ \mathcal{O}_{\mathcal{F}\cap  \mathrm{N}_{G}[H]} \mathrm{N}_{G}[H]}(M,N) \rightarrow \mathrm{Ext}^{i+1}_{\mathcal{O}_{\mathcal{H}}G}(M,N) \rightarrow \ldots \ ,  \]
and a long exact sequence of Tor functors
\[  \ldots \rightarrow \bigoplus_{[H] \in \mathcal{I}} \mathrm{Tor}_{i}^{ \mathcal{O}_{\mathcal{F}\cap  \mathrm{N}_{G}[H]} \mathrm{N}_{G}[H]}(M,T) \rightarrow  \Big(\bigoplus_{[H] \in \mathcal{I}} \mathrm{Tor}_{i}^{\mathcal{O}_{\mathcal{F}[H]}\mathrm{N}_{G}[H]}(M,T)\Big)\oplus  \mathrm{Tor}_{i}^{\mathcal{O}_{\mathcal{F}}G}(N,T) \rightarrow   \]
\[ \mathrm{Tor}_{i}^{\mathcal{O}_{\mathcal{H}}G}(M,T) \rightarrow  \bigoplus_{[H] \in \mathcal{I}} \mathrm{Tor}_{i-1}^{ \mathcal{O}_{\mathcal{F}\cap  \mathrm{N}_{G}[H]} \mathrm{N}_{G}[H]}(M,T) \rightarrow \ldots \ . \]
\end{theorem}
\begin{proof} Using Lemma \ref{lemma: exact and project} and the adjointness between induction and restriction (\ref{eq: adjoint}), one checks that the desired long exact sequences are the long exact sequences induced by $\mathrm{Ext}^{\ast}_{\mathcal{O}_{\mathcal{H}}G}(-,N)$ and $\mathrm{Tor}_{\ast}^{\mathcal{O}_{\mathcal{H}}G}(-,T)$ from the short exact sequence of coefficients of Proposition \ref{prop: short exact seq} applied to $M$.
\end{proof}

\end{document}